\documentclass{amsart}
\usepackage{amsmath,amssymb,amsthm,enumerate,hyperref,xcolor,enumitem,tikz,tikz-cd,biblatex,graphicx, quiver, adjustbox, biblatex}
\newtheorem{definition}{Definition}[subsection]
\newtheorem{proposition}{Proposition}[subsection]
\newtheorem{example}{Example}[subsection]
\newtheorem{theorem}{Theorem}[subsection]

\newcommand{\Hom}{\mathrm{Hom}}

\title{Intuitionistic logic, dual intuitionistic logic, and modality}
\author{Safal Raman Aryal}

\begin{document}
\begin{abstract}
    We explore various semantic understandings of dual intuitionistic logic by exploring the relationship between co-Heyting algebras 
    and topological spaces. First, we discuss the relevant ideas in the setting of Heyting algebras and intuitionistic logic, 
    showing organically the progression from the primordial example of lattices of open sets of topological spaces to more general 
    ways of thinking about Heyting algebras. Then, we adapt the ideas to the dual intuitionistic setting, and use them to prove a 
    number of interesting properties, including a deep relationship both intuitionistic and dual intuitionistic logic share through Kripke semantics
    to the modal logic $\mathsf{S4}$. 
\end{abstract}
\maketitle
\tableofcontents
\section*{Introduction}
In this paper, we aim to explore intuitionistic logic and its dual from an ``immanent'' mathematical perspective. The approach we will take is different from the 
common route of directly constructing Kripke semantics: rather, we will exploit functorial dualities that exist between certain classes
of topological and algebraic objects, and use them to characterize the features of dual intuitionistic logic before moving to Kripke semantics. Before results on dual intuitionistic 
logic are explained, we will spend some time developing the relation between topological and algebraic semantics for intuitionistic logic by 
considering the relationship between Heyting algebras and lattices of open subsets of a topological space. Our development of dual intuitionistic logic will begin with analysis of the lattice of 
\textit{closed} sets of a topological space. We hope this point of view helps the reader compare and contrast the features of these 
two logics, especially given the paraconsistent characteristics of dual intuitionistic logic.
In a set of concluding remarks, we will discuss the relationship between our topological view of these logics, and the possible worlds 
semantics they can be given, and show that these two semantic characterizations of the logics can be subsumed under the Kripke semantics 
of $\mathsf{S4}$ modal logic.

We will begin our exposition concretely, by exploring the topological roots of the Heyting algebra. This point of view motivates the 
\textit{spatial} view of intuitionistic (and dual intuitionistic) logic we wish to emphasize throughout this paper, and complements
the abstract algebraic methods used in a natural way. Indeed, it is the point of view of the author that since Stone's Representation Theorem,
topological methods have been an indispensable to understanding a wide range of algebraic structures in an ``immanent'' and effective way.
The rich dualities and correspondences that exist are a testament to the power of functorial mathematics. It is the hope of the 
author that discussing these correspondences in the context of \textit{logical} systems and in relation to modal logic will prove 
both mathematically and philosophically rich. Indeed, while much of this material is already known, the author believes that the \textit{framing}
provided here -- specifically the connection that both logics share to $\mathsf{S4}$ -- is relatively novel, though Zolfaghari and Reyes provide 
a similar proof. \cite{reyes_1996_biheyting} 

On this account, most of the results will look familiar, but perhaps the proofs and connections emphasized will not. Most propositions as stated are necessary for the 
structure of exposition as adopted in the paper, and synthesize commonly known results throughout the literature in ways that are not usually 
seen. Thus, as the author has not directly copied propositions or proofs (with the exception of the proof of Stone's Representation Theorem),
propositions should be understood as being proven with a synthesis of the resources given in the references of the paper. However, references that 
were used to select results to prove for each section (or to introduce technical tools) are given at the beginning of each section. 

This paper comes out of an undergraduate research project done at Tufts University on non-classical logics, funded by the Tufts Summer Scholars grant.
Details can be found in the Acknowledgements section.

\section{Open sets, Heyting algebras, and intuitionistic logic}
The exposition in this section will re-prove a number of results from Grätzer \cite{gratzer_2017_lattice} and Johnstone \cite{johnstone_1982_stone}. 
Basic topological facts can be found in Munkres \cite{munkres_2018_topology}, likewise, basic facts about algebra are recalled mirroring the 
exposition from Aluffi \cite{paoloaluffi_2021_algebra}, though that text does not deal with the structures at issue in this paper.

Let $X$ be a set. A \textit{topology} on $X$ is a collection of subsets $\mathcal{O}(X) \subseteq \mathcal{P}(X)$, with elements called \textit{open sets}, satisfying
the following properties
\begin{definition}
    Topology
\end{definition}
A \textit{topology} $\mathcal{O}(X) \subseteq \mathcal{P}(X) $ on a set $X$ satisfies the following properties:
\begin{enumerate}
    \item $\emptyset, X \in \mathcal{O}(X)$
    \item $\mathcal{O}(X)$ is closed under finite intersections
    \item $\mathcal{O}(X)$ is closed under arbitrary unions
\end{enumerate}

If $X$ is a set that has a topology defined on it, then we may refer to it as a \textit{topological space}. Here, we concern ourselves
with the order structure of the topology $\mathcal{O}(X)$ itself. Given that $\mathcal{O}(X)$ is a collection of subsets of $X$, containing 
$\emptyset$ and $X$, it seems natural to induce this order with \textit{inclusion}, where for $A, B \in \mathcal{O}(X)$, $A \leq B \iff 
A \subseteq B$. We will now explore some properties of $\mathcal{O}(X)$ with this partial order put on it. It turns out that this order 
structure gives rise to a rich \textit{logical} structure worth characterizing in detail: indeed, it turns out to be a \textit{lattice}.

\begin{definition}
    A \textit{lattice} $L$ is a partially ordered set in which for any two elements $a, b \in L$, there exists a unique \textit{supremum},
    denoted $a \vee b$, and \textit{infimum}, denoted $a \wedge b$.\end{definition}
\begin{proposition}
    For any topology $\mathcal{O}(X)$ on a set $X$, $(\mathcal{O}(X), \subseteq)$ is a bounded distributive lattice.
\end{proposition}
\begin{proof}
    First we show that $\mathcal{O}(X)$ is a lattice.
    Let $A, B \in \mathcal{O}(X)$, with order given by inclusion. Define $A \vee B := A \cup B$ and $A \wedge B := A \cap B$. Since $\mathcal{O}(X)$ is closed under 
    arbitrary unions and finite intersections, $A \vee B$ and $A \wedge B$ are guaranteed to exist and are the supremum and infimum of the two objects
    relative to this order respectively, because the union is the smallest set with both $A$ and $B$ as subsets, and the intersection is the largest shared subset 
    of $A$ and $B$. Thus, $\mathcal{O}(X)$ is a lattice. To show that $\mathcal{O}(X)$ is distributive, it suffices to note that $A \vee (B \wedge C) = (A \vee B) \wedge 
    (A \vee C)$ and $A \wedge (B \vee C) = (A \wedge B) \vee (A \wedge C)$ hold by definition from the distributivity of unions 
    and intersections for sets. Finally, note that for any open set $A \in \mathcal{O}(X)$, we have that $A \subseteq X$, and $\emptyset \subseteq A$.
    Because the axioms for a topology guarantee the inclusion of $X$ and the empty set, observe that $X \vee A = X$ and $\emptyset \wedge 
    A = \emptyset$, thus $X$ acts as a top element, and $\emptyset$ as a bottom element, meaning $\mathcal{O}(X)$ is bounded.
\end{proof}

Being a distributive lattice already gives $\mathcal{O}(X)$ a lot of interesting structure. However, there is actually \textit{more} structure we can 
give to $\mathcal{O}(X)$ -- it is possible to define an ``implication'' operator on this topology, treating open sets as logical propositions, and 
operations on open sets as logical operators (indeed, as we will see, \textit{intuitionistic} operators).

\begin{definition}
    Let $L$ be a bounded distributive lattice. Then for any $a, b \in L$, we define the \textit{Heyting implication} $a \rightarrow b$ as follows 
\end{definition}
\begin{align*}
    a \rightarrow b := \bigvee\{x \in L \mid a \wedge x \leq b\}.
\end{align*}
We will now use this definition to prove another property of the lattice $\mathcal{O}(X)$.
\begin{proposition}
    For any $A, B \in \mathcal{O}(X)$, the Heyting implication $A \rightarrow B$ always exists in $\mathcal{O}(X)$.
\end{proposition}
\begin{proof}
    Let $A, B \in \mathcal{O}(X)$. The Heyting implication is defined as $A \rightarrow B := \bigvee\{C \in \mathcal{O}(X) \mid A \wedge C \leq B\}$.
    Knowing that the partial order on $\mathcal{O}(X)$ is given by $\subseteq$, we can think of $A \rightarrow B$ as a union over
    the set of all open sets $C \in \mathcal{O}(X)$ for which the intersection $A \cap C$ is a subset of $B$. If $A \cap C = \emptyset$, then since 
    $\emptyset \in \mathcal{O}(X)$, and $\emptyset \subset B$, $A \rightarrow B$ is trivially non-empty (it contains the empty set). On the 
    other hand, suppose $A \cap C$ is non-empty. Put $C = A \cap B$. Then, $A \cap (A \cap B) \subseteq B$ is trivially true, and 
    thus $A \rightarrow B$ is non-empty in this case as well. This completes the proof that Heyting implications always exist in $\mathcal{O}(X)$.
\end{proof}
With this, we have formally defined a notion of one open set ``implying'' another. However, it is not intuitively clear what this means -- so, let us 
try to develop a more concrete way of understanding this phenomenon.
Let $X$ (the whole space) refer to a ``top'' truth value (1, or simply ``true''), and let $\emptyset$ refer to the ``bottom'' truth value 
(0, or ``false''). If $A \rightarrow B = X$, then $A \cap X \subseteq B$. Since $X$ is the whole space, this implies that $A \subseteq B$, thus, 
$A \leq B$ in the lattice. Conversely, if $A \leq B$, then $A \subseteq B$, thus for any $C \in \mathcal{O}(X)$, $A \cap C \subseteq B$ holds true.
This, of course, means that $A \cap X \subseteq B$, and thus $A \rightarrow B = X$. In other words, $A \rightarrow B$ \textit{if and only if}
$A \leq B$ in $\mathcal{O}(X)$. 

The above realization gives us a spatially intuitive way to think about implication in the ``logic'' of these open sets interpreted as propositions.
For any propositions $A$ and $B$, $A$ implies $B$ if and only if $A$ is a subset of $B$. In other words, if you're in $A$, you're in $B$ by default,
and there's no way to \textit{leave} it, which is exactly what it means for $A$ to imply $B$. If $A \rightarrow B$ and $B \rightarrow A$ are both true, 
this means that $A$ and $B$ occupy exactly the \textit{same} region of the space $X$. This gives us a nice spatial intuition for logical equivalence.
Given that unions correspond to joins and intersections correspond to meets, we can think of logical disjunction and conjunction as \textit{joining}
and \textit{intersecting} regions of space respectively.

We still have yet to characterize the way negation works in this logic of open sets. In classical logic, the negation of a proposition $p$ is defined 
as $\neg p := p \rightarrow 0$. Let us adapt this definition and study it in our context. Consider some open set $A \in \mathcal{O}(X)$. By the classical 
definition, we have $\neg A = \bigcup \{U \in O \mid A \cap U = \emptyset\}$ -- in other words, it is the maximal open set whose intersection 
with $A$ is (a subset of) the empty set. This set admits a special topological characterization: it is the \textit{interior of the complement} of 
$A$. Let us introduce some definitions to make the meaning of this clear.

\begin{definition}
    The \textit{interior} of an open set $A \in \mathcal{O}(X)$, denoted $\mathrm{int}(A)$, is the largest open subset $S \subseteq A$.
\end{definition}
\begin{definition}
    The \textit{complement} $A^c$ of a set $A \in \mathcal{O}(X)$ is the set of all points $x \in X$ in the space such that $x \not \in A$.
    The complement of an open set is called a \textit{closed} set, thus $A^c$ is closed for any $A \in \mathcal{O}(X)$.
\end{definition}
The original motivation for the concepts of open and closed sets came from mathematical analysis -- given the uncountable infinitude 
of the real numbers, one could think of either closed sets containing ``boundary points'', or open sets that did \textit{not} contain 
them. While this concept formalizes only very loosely, the intuition is crucial for understanding the concept of a \textit{closure}. 
Consider the open set $(0, 1) \subset \mathbb{R}$, which does not contain its boundary points of 0 and 1. The corresponding \textit{closed}
set $[0, 1]$ does contain them -- and since the original open set was missing precisely those two points, this is also the \textit{smallest}
closed set containing $(0, 1)$. This is precisely what motivates our next definition.
\begin{definition}
    The \textit{closure} $\mathrm{Cl}(A)$ of a set $A \subseteq X$ is the smallest closed set $A'$ such that $A \subseteq A'$. Formally,
\end{definition}
\begin{align*}
    \mathrm{Cl}(A) = \bigcap \{A' \subseteq X \mid A \subseteq A\}.
\end{align*}
We now return to discussing negation. Given an open set $A$, the negation $\neg A$ is the maximal open set disjoint from $A$. How do 
we obtain this? Of course, we know that $A^c$ is the maximal subset of the space $X$ disjoint from $A$, but this set is not necessarily 
open. However, we know that $A^c$ must be the \textit{closure} of $\neg A$, and so given that the concepts of interior and closure as 
defined above are dual to one another, we see that $\neg A = \mathrm{int}(A^c)$. This definition of negation has an interesting consequence,
which cements the fact that the logic of $\mathcal{O}(X)$ is definitely not classical.

\begin{proposition}
    For $A \in \mathcal{O}(X)$, $A = \neg \neg A$ does not necessarily hold.
\end{proposition}
\begin{proof}
    Consider the space $X = \{a, b, c\}$, with the following topology
    \begin{align*}
        \mathcal{O}(X) = \{\emptyset, \{a\}, \{a, b\}, X\}.
    \end{align*}
    This is not the discrete topology, as not every set is open. Put $A = \{a, b\}$. Then, we have $A^c = \{c\}$, and since this 
    is not an open set, $\neg A = \mathrm{int}(A^c) = \{\emptyset\}$. Since $\neg A$ is empty, we have $\neg A^c = X$. However, because $X$ 
    is open, we see that $\mathrm{int}(\neg A^c) = \mathrm{int}(X) = X$, which completes the proof that $A \neq \neg \neg A$ is not always true in the 
    logic of open sets.
\end{proof}
These considerations underline the intuitionistic nature of the logic of open sets. The failure of double negation to generally imply 
equality reflects a logic that is fundamentally different from classical logic. Instead, the lattice $\mathcal{O}(X)$ of open sets embodies
an algebraic structure known as a \textit{Heyting algebra}, which provides a complete semantics for intuitionistic logic. This algebraic 
perspective allows us to interpret conjunction, disjunction, implication, and negation in a way that aligns with the topological behavior 
of open sets.

\begin{definition}
    A \textit{Heyting algebra} $H$ is a bounded distributive lattice characterized by the condition that if $a, b \in H$, then $a \rightarrow b \in H$.
\end{definition}

It is obvious that the lattice $\mathcal{O}(X)$ of open sets of a topological space $X$ that we have been discussing until now is a 
Heyting algebra. Every Heyting algebra is a model of a propositional theory in intuitionistic logic -- however, the proof of 
this is very involved, and given our main focus is paraconsistent logic, we will be omitting it. We will, however, sketch an 
\textit{intuitive} picture of why this is true.

\begin{definition}
    An \textit{algebraic model} of intuitionistic propositional theory $T$ is a triple $\mathbf{M} = \langle H, T, V\rangle$, where $H$ 
    is a Heyting algebra, $T$ is the theory in question, and $V : T \rightarrow H$ is a function assigning formulae in $T$ to elements 
    of $H$. Suppose $T$ is formulated in the language $\mathcal{L} = \{\top, \bot, \vee, \wedge\ , \rightarrow\}$. Let $\varphi \in T$. 
    We may then define $[\varphi]_\mathbf{M}$ inductively as follows:
    \begin{enumerate}
        \item $[\bot]_\mathbf{M} = 0$
        \item $[p]_\mathbf{M} = V(p)$
        \item $[\varphi \vee \psi]_\mathbf{M} = [\varphi]_\mathbf{M} \vee [\psi]_\mathbf{M}$
        \item $[\varphi \wedge \psi]_\mathbf{M} = [\varphi]_\mathbf{M} \wedge [\psi]_\mathbf{M}$
        \item $[\varphi \rightarrow \psi]_\mathbf{M} = [\varphi]_\mathbf{M} \rightarrow [\psi]_\mathbf{M}$
    \end{enumerate} 
    \end{definition}
You may notice that while this algebraic model is well-defined, the syntactic rules of intuitionistic logic and the algebraic 
operations of Heyting algebras are so similar that they are hard to distinguish: for this reason, algebraic logic has often been 
criticized as ``syntax in disguise''. However, just as we could ``algebraize'' the open sets $\mathcal{O}(X)$ of a topological space $X$, 
it is also possible to ``topologize'' a Heyting algebra $H$ and thus obtain a topological semantics for intuitionistic logic.

This process is at the heart of the arguments that Marshall Stone used to prove his celebrated representation theorems for 
various kinds of distributive lattices (including Boolean algebras and Heyting algebras). Let us develop some formalism to 
help articulate these ideas.
\begin{definition}
    A \textit{filter} on a lattice $L$ is a subset $F \subseteq L$ such that if $a \in F$, and $a \leq a'$, then $a' \in F$, and 
    if $a, b \in F$, then $a \wedge b \in F$. \end{definition} 
\begin{definition}
    A filter $F \subseteq L$ is \textit{prime} when $a \vee b \in F$ implies either $a \in F$ or $b \in F$.
\end{definition}
Intuitively, we can think of filters as \textit{upward closed} subsets of a lattice, and prime filters as those which are ``focused''
enough to have join-reducible elements incorporate their constitutents into the filters too. It turns out that the set of \textit{all}
prime filters on a Heyting algebra has some interesting topological structure.
\begin{proposition}
    Let $H$ be a Heyting algebra, and let $X$ be the set of all prime filters on $H$. For $h \in H$, let 
    \begin{align*}
        \beta(h) := \{A \in X \mid h \in A\}.
    \end{align*}
    Then $\mathcal{B} := \{\beta(h) \mid h \in H\}$ generates a topology on $X$.
\end{proposition}
\begin{proof}
    We will proceed by proving constructively that $\mathcal{B}$ satisfies the properties of a basis for a topology.
    Let $A \in X$. We must find an open set in $\mathcal{B}$ that covers $A$. By definition of a prime filter, $A$ is not empty, which 
    means we may pick some $a \in A$. Then, by construction, the open set $\beta(a) \in \mathcal{B}$ covers $A$. 

    Now, let $\beta(a), \beta(b) \in \mathcal{B}$, and $A \in \beta(a) \cap \beta(b)$. Then $a \in A$ and $b \in A$, and thus 
    $a \wedge b \in A$ by definition of a prime filter. By construction, we have that $A \in \beta(a \wedge b)$. It is clear that 
    $\beta(a \wedge b) = \beta(a) \cap \beta(b)$ as sets: filters in the set on the LHS must contain $a$ and $b$ by the upward
    closure of filters, whereas filters on the RHS must contain $a \wedge b$ by the definition of a prime filter.

    Since we may cover any prime filter of $H$ with an open set in $\mathcal{B}$, and since any prime filter covered by the intersection of some 
    open sets in $\mathcal{B}$ may be covered by another open set which is a equal to that intersection, we may conclude that 
    $\mathcal{B}$ is a basis for a topology on $X$, and thus generates a topology on $X$.
\end{proof}
Now we may prove the following theorem.
\begin{theorem}
    Stone's representation theorem for Heyting algebras. 

    Let $H$ be a Heyting algebra. Then $H$ embeds into the lattice of open sets $\mathcal{O}(X)$ for a topological space $X$.
\end{theorem}
\begin{proof}
    From Proposition 1.0.5, we know that there is a topology on the set $X$ of all prime filters of $H$ generated by the basis 
    $\mathcal{B} = \{\beta(h) \mid h \in H\} \subseteq \mathcal{O}(X)$. Define a map $\varphi : H \rightarrow \mathcal{O}(X)$
    with $\varphi(h) = \beta(h)$ for $h \in H$. Then, we have for $a, b \in H$ that 
    \begin{align*}
        \varphi(a \wedge b) = \beta(a \wedge b) = \beta(a) \cap \beta(b) \\
        \varphi(a \vee b) = \beta(a \vee b) = \beta(a) \cup \beta(b). \\
    \end{align*}
    Thus $\varphi$ is an injective homomorphism of lattices, which completes the proof.
\end{proof}

Thus, any Heyting algebra embeds into a \textit{subset} of the lattice of open sets of some topological space. If the Heyting algebra 
in question is finite, this embedding is actually an isomorphism. It is for \textit{infinite} lattices that we run into a problem: 
suppose we have a countably infinite sized Heyting algebra -- the prime filters on it could generate an \textit{uncountably} large 
lattice of open sets, making a surjection impossible to construct. Nevertheless, the theorem reassures us that an \textit{embedding}
can always exist, and this suffices for our purposes: from \textit{any} Heyting algebra, we can always obtain a sublattice of open sets of a 
topological space. This structure is complete as a semantics for propositional intuitionistic logic too (see Bezhanishvili and Holliday \cite{BEZHANISHVILI2019403} for a 
proof). Thus, we may ``topologize'' any Heyting algebra in this way to study it. It is then easy to see that if some formula of 
intuitionistic propositional logic $\varphi$ holds in $H$, it must hold in $H \mapsto \beta(H)$ as well, and thus the topological 
semantics we obtain are complete. Given this, we are once more able to freely exchange algebraic and topological notions 
in our analysis. 

Let us proceed to understand how we might assign \textit{truth values} to elements of a Heyting algebra considered 
as propositions. We introduce some formalism.
\begin{definition}
    A \textit{homomorphism} of Heyting algebras $H$ and $K$ is a map $\varphi : H \rightarrow K$ that preserves joins, meets and 
    Heyting implications (i.e $\varphi(a \vee b) = \varphi(a) \vee \varphi(b)$, and similarly for the other operations).
\end{definition}
\begin{definition}
    Let $H$ be a Heyting algebra. A \textit{truth value assignment} (which we will also call \textit{extensional homomorphisms})
    is a surjective Heyting algebra homomorphism $\varphi : H \rightarrow 
    \{0, 1\}$.
\end{definition}
\begin{definition}
    Let $\varphi : H \rightarrow K$ be a homomorphism of Heyting algebras (if $K \neq \{0, 1\}$, we may also call this an
    \textit{intensional homomorphism}). Then, the \textit{kernel} of the homomorphism is defined as follows
    \begin{align*}
        \ker \varphi = \{h \in H \mid \varphi(h) = 0_K\}.
    \end{align*}
\end{definition}
Logically speaking, we can interpret homomorphisms $\varphi : H \rightarrow K$ of Heyting algebras as ways to take $K$ as a 
semantic model for the sentences in $H$. The kernel $\ker \varphi$ can be thought of as the set of all $H$-sentences that are 
\textit{falsified} by $K$ (indeed, that is what mapping to $0_K$ \textit{is}, in a sense). If $\varphi$ maps all of $H$ onto $0_K$,
we say it is the \textit{trivial} homomorphism. The existence of a nontrivial homomorphism between two Heyting algebras implies 
that there is a way to logically interpret the domain's theory in the codomain. Conversely, if the only homomorphism that exists between 
two Heyting algebras is trivial, there is \textit{no} way to interpret the domain's theory in the codomain. The formalism of homomorphisms
also allows us to construct more intricate sorts of relationships between intuitionistic propositional theories.
\begin{definition}
    A \textit{ideal} on a lattice $L$ is a subset $I \subseteq L$ such that if $a \in I$, and $a \geq a'$, then $a' \in I$, and 
    if $a, b \in I$, then $a \vee b \in I$. 
\end{definition}
It is easy to see that the notion of an ideal is \textit{dual} to that of a filter. It is actually possible to construct Heyting algebras 
from ideals or filters that \textit{validate} or \textit{falsify} certain sentences, due to their downward-closed and upward-closed 
natures respectively. This is done with the machinery of quotient maps.

\begin{definition}
    Let $H$ be a Heyting algebra, and let $I$ be an ideal of $H$. Then the \textit{quotient} $H/I$ is the set of all equivalence classes 
    on elements of $H$ relative to the relation $R$, where $xRy$ when $x \rightarrow y \in I \iff y \rightarrow x \in I$ for $x, y \in H$.
    It is easy to see that $H/I$ is a Heyting algebra, with joins and meets defined as unions and intersections of equivalence classes.
\end{definition}
\begin{proposition}
    Let $H$ be a Heyting algebra, and let $I$ be an ideal of $H$.
    Then the quotient map $\varphi : H \rightarrow H/I$ is a surjective homomorphism of Heyting algebras.
\end{proposition}
\begin{proof}
    Define $\varphi : H \rightarrow H/I$ with $\varphi(x) = [x]_R$ for $x \in H$, where $x R y \iff (x \rightarrow y) \wedge (y \rightarrow x) \in I$. 
    Then $\varphi$ is surjective. We will show $\varphi$ is a homomorphism of Heyting algebras. First, notice joins and meets are preserved
    \begin{align*}
    \varphi(x \vee y) &= [x \vee y]_R = [x]_R \cup [y]_R = \varphi(x) \vee \varphi(y), \\
    \varphi(x \wedge y) &= [x \wedge y]_R = [x]_R \cap [y]_R = \varphi(x) \wedge \varphi(y).
    \end{align*}
    For implication, observe that $\varphi(x \rightarrow y) = [x \rightarrow y]_R$. Since $H/I$ is a Heyting algebra, the equivalence 
    class $[z]_R$ for any $z$ satisfies:
    $$
    [z]_R \leq [x]_R \rightarrow [y]_R \iff z \rightarrow (x \rightarrow y) \in I.
    $$
    Thus, $\varphi(x \rightarrow y) = \varphi(x) \rightarrow \varphi(y)$. Therefore, $\varphi$ is a surjective homomorphism of Heyting algebras,
    which completes the proof.
\end{proof}
One can take quotients by \textit{filters} on a Heyting algebra as well, and define the quotient map in a symmetric way. Taking a 
quotient by an ideal $I \subseteq H$ \textit{collapses} all the elements of the ideal onto the bottom element $0_H$, whereas taking a quotient 
by a filter collapses the elements of the ideal onto the \textit{top} element $1_H$. In logical terms, we can think of quotient structures 
as Heyting algebras with filters or ideals to be reflecting how the \textit{rejection} or \textit{affirmation} impacts other sentences 
of the theory. Thus, often, this quotient structure will be \textit{simpler} than the original structure. Thinking of filters as sets 
of ``provable'' statements and ideals as sets of ``unprovable'' statements, we may thus ``factor'' Heyting algebras $H$ by quotient 
structures as follows, assuming we have ideals $I_0, \dots, I_n$ and filters $F_0, \dots, F_m$ in $H$.
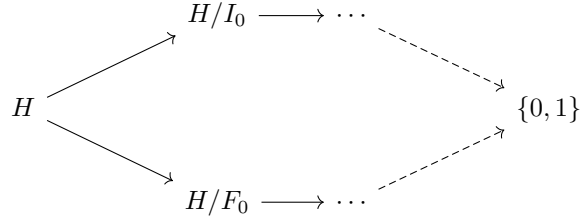
\begin{center}
\[\begin{tikzcd}
	&& {H/I_0} & \cdots \\
	H &&&&& {\{0, 1\}} \\
	&& {H/F_0} & \cdots
	\arrow[from=1-3, to=1-4]
	\arrow[dashed, from=1-4, to=2-6]
	\arrow[from=2-1, to=1-3]
	\arrow[from=2-1, to=3-3]
	\arrow[from=3-3, to=3-4]
	\arrow[dashed, from=3-4, to=2-6]
\end{tikzcd}\]
\end{center}
Such a diagram can be drawn for any choice of starting ideal or filter. In fact, the diagram shows that we could even \textit{alternate} 
between taking quotients with ideals and filters, because we are guaranteed to reach the two-element Heyting algebra with a sequence of 
either just ideal or filter quotients (an inductive argument is naturally made given this). Logically speaking, because ideals and 
filters are directionally closed, we see that their members always map to either 0 (for ideals)
or 1 (for filters) in homomorphisms from their parent lattices to $\{0, 1\}$ -- in other words, ideals are sets of \textit{false} sentences, and 
filters are sets of \textit{true} sentences (with intuitionistic interpretations of the truth values, of course), all closed under implication.
The diagram above suggests it is always possible to \textit{collapse} such sets to single elements of a Heyting algebra, thus simplifying logical 
theories -- indeed, until one obtains the trivial theory.

It is of course possible to construct quotient Heyting algebras with \textit{arbitrary} equivalence relations that respect the Heyting algebra 
structure (if the quotient maps do not respect this structure, there's a possibility their images \textit{won't} be Heyting algebras). Such 
equivalence relations are called \textit{congruences}. 
\begin{definition}
    An equivalence relation $R$ on a lattice $L$ is a \textit{congruence} if it respects the operations of the lattice
    ,i.e $aRb$ and $cRd$ imply $(a \wedge c) R (b \wedge d)$, and similarly for other operations, for $a, b, c, d \in L$.
\end{definition}

It is easy to see that taking quotients by congruences allows for the construction of new Heyting algebras by collapsing certain formulae 
as logically \textit{equivalent} -- even if they aren't intuitionistically equivalent, the fact that the congruence in an equivalence relation 
allows for a ``simulation'' of if they were. 

\section{Abstract duality and functors}

We have spent considerable time discussing maps between Heyting algebras that respect their \textit{algebraic} structure, and how they 
can be used to construct new Heyting algebras from old ones. How should we topologize these notions? \textit{Can} they be topologized?
It turns out that Heyting algebra homomorphisms and maps of topological spaces share a complicated relationship, which 
we will now characterize, continuing to follow Johnstone. Basic facts about category theory are taken from Riehl \cite{riehl_2016_category}. 
\begin{definition}
    Let $(X, \sigma)$ and $(Y, \tau)$ be tuples of topological spaces with their associated topologies. A \textit{continuous map} is a 
    function $f : X \rightarrow Y$ such that $f^{-1}(U) \in \sigma$ is open for any $U \in \tau$. 
\end{definition}

\begin{definition}
    For $(X, \sigma)$ and $(Y, \tau)$, a function $f : X \rightarrow Y$ is an \textit{open map} if for any $U \in \sigma$, 
    $f(U)$ is open in $Y$ (i.e $f(U) \in \tau$).
\end{definition}

\begin{proposition}
    Let $\varphi : H \rightarrow K$ be a homomorphism of Heyting algebras, and suppose $H \hookrightarrow \mathcal{O}(X)$ 
    and $K \hookrightarrow \mathcal{O}(Y)$, for topological spaces $X$ and $Y$ by Theorem 1.0.1. Then, $\varphi$ induces a 
    continuous map $f : Y \rightarrow X$.
\end{proposition}
\begin{proof}
    Recall that the points of $X$ are prime filters on $H$, and the points of $Y$ are prime filters on $K$. Define 
    $f : Y \rightarrow X$ with $f(P) = \varphi^{-1}(P)$, for a prime filter $P \in Y$ associated with the lattice $K$.
    Since the preimage of a prime filter is also a prime filter, we have that $f(P)$ is prime, thus $f$ is well-defined.

    We proceed by proving that the preimage under $f$ of an open set in $\tau$ is open. Let $\beta(h) \in \sigma$. 
    By the formalism developed in our proof of Stone's representation theorem, we have $\beta(h) = \{P \in X \mid h \in P\}$.
    Take the preimage $f^{-1}(\beta(h)) = \{Q \in Y \mid f(Q) \in \beta(h)\}$. Since $f(Q) = \varphi^{-1}(Q)$, we may write 
    $f^{-1}(\beta(h)) = \{Q \in Y \mid \varphi^{-1}(Q) \in \beta(h)\}$. Since $\varphi^{-1}(Q) \in \beta{h}$ implies that 
    $h \in \varphi^{-1}(Q)$, we have $f^{-1}(\beta(h)) = \{Q \in Y \mid h \in \varphi^{-1}(Q)\}$. By the definition of a preimage 
    under $\varphi$, we see that $h \in \varphi^{-1}(Q)$ if and only if $\varphi(h) \in Q$. Thus, we may write 
    \begin{align*}
        f^{-1}(\beta(h)) = \{Q \in Y \mid \varphi(h) \in Q\} = \beta(\varphi(h)).
    \end{align*}
    Put $Y = \beta(\varphi(h))$. Then $Y$ is precisely the open set corresponding to the element $\varphi(h) \in K$. We conclude that 
    the preimage of an open set under $f$ is open.
\end{proof}
It is important to note that not every continuous map of topological spaces induces a Heyting algebra homomorphism, as there is no
requirement that continuous maps preserve \textit{implication}. However, a continuous map will always induce a \textit{lattice} 
homomorphism (in fact, a \textit{frame} homomorphism) between the lattices of open sets of the topological spaces in its preimage 
and image, i.e. if $f : Y \rightarrow X$ is the continuous map in question, then we get the lattice homomorphism 
$\varphi : \mathcal{O}(X) \rightarrow \mathcal{O}(Y)$. There is a condition for a continuous map inducing a Heyting algebra homomorphism, 
which we will prove now.
\begin{proposition}
    A continuous map $f : Y \rightarrow X$ induces a Heyting algebra homomorphism $\varphi : \mathcal{O}(X) \rightarrow \mathcal{O}(Y)$
    if and only if $f$ is an open map.
\end{proposition}
\begin{proof}
    Let $f: Y \rightarrow X$ be a continuous map, and let $\varphi: \mathcal{O}(X) \rightarrow \mathcal{O}(Y)$ be the induced map defined by $\varphi(U) = f^{-1}(U)$ for $U \in \mathcal{O}(X)$.

    Suppose $\varphi$ is a Heyting algebra homomorphism. To show that $f$ is an open map, consider any open set $V \subseteq Y$. We need to show that $f(V)$ is open in $X$.
    Since $\varphi$ is a Heyting algebra homomorphism, it preserves finite meets (i.e. intersections) and finite joins (i.e. unions). Importantly, it also preserves the implication operation $\rightarrow$ defined in the Heyting algebra structure.
    Let $U = f(V)$, where $V$ is an open set in $Y$. To show that $U$ is open in $X$, take any $x \in U$. Then $x = f(y)$ for some $y \in V$. Consider the open set $V_y = V$ and the corresponding open set $U_y = f(V_y)$ in $X$.
    The crucial observation here is that the Heyting implication $U \rightarrow V_y$ is defined by the largest open set $W$ such that $W \cap U \subseteq V_y$. Since $\varphi$ is a Heyting algebra homomorphism, it preserves this implication operation.
    Thus, $\varphi(U \rightarrow V_y) = \varphi(U) \rightarrow \varphi(V_y)$, and by continuity of $f$, $\varphi(U)$ is open in $Y$. Hence, $f(V_y)$ is open in $X$, and $U = f(V)$ is open in $X$. Therefore, $f$ is an open map.

    Conversely, suppose $f$ is an open map. We need to show that the induced map $\varphi: \mathcal{O}(X) \rightarrow \mathcal{O}(Y)$ is a Heyting algebra homomorphism.
    Since $f$ is open, for any open set $U \subseteq X$, $f^{-1}(U)$ is open in $Y$. The map $\varphi$ defined by $\varphi(U) = f^{-1}(U)$ preserves finite meets (intersections) and finite joins (unions) by the properties of preimages under continuous maps.
    Moreover, the implication operation $U \rightarrow V$ in $\mathcal{O}(X)$ corresponds to the largest open set $W$ such that $W \cap U \subseteq V$. Since $f$ is open, the preimage $\varphi(U \rightarrow V)$ is the largest open set $W'$ in $\mathcal{O}(Y)$ such that $W' \cap f^{-1}(U) \subseteq f^{-1}(V)$, and thus $\varphi(U \rightarrow V) = \varphi(U) \rightarrow \varphi(V)$.
    Therefore, $\varphi$ preserves the implication operation, and thus $\varphi$ is a Heyting algebra homomorphism.
    This completes the proof that $f$ induces a Heyting algebra homomorphism $\varphi$ if and only if $f$ is an open map.
\end{proof}
Now, we are finally at a place where we can conclude our discussion on intuitionistic logic by considering, in full generality, how 
its algebraic and topological representations interact with one another. We have described the conditions under which continuous maps 
induce Heyting algebra homomorphisms, and vice versa. The language of categories and functors gives us a systematic language which 
lets us discuss in great detail exactly \textit{how} these conditions interact -- it does this by giving us a way to rigorously 
discuss various kinds of mathematical objects mapped onto each other in different ways.
\begin{definition}
    A \textit{category} $\mathcal{C}$ consists of a collection of objects $O$, and maps $M$ between them, called \textit{morphisms}.
    Identity morphisms from objects to themselves always exist, and composition of morphisms is associative.
\end{definition}
\begin{definition}
    Let $\mathcal{C}$ and $\mathcal{D}$ be categories. A \textit{functor} $F : \mathcal{C} \rightarrow \mathcal{D}$ maps objects of 
    $\mathcal{C}$ to objects of $\mathcal{D}$, and morphisms of $\mathcal{C}$ to morphisms of $\mathcal{D}$, such that 
    $F(\text{id}_X) = \textit{id}_{F(x)}$ for objects $X$ of $\mathcal{C}$, and $F(f \circ g) = F(f) \circ F(g)$, which makes $F$ a \textit{covariant}
    functor, or $F(f \circ g) = F(g) \circ F(f)$, which makes $F$ a \textit{contravariant} functor. A functor's contravariance will be 
    denoted by writing it as $F : \mathcal{C} \rightarrow \mathcal{D}^{op}$. 
\end{definition}

Categories and functors are the perfect conceptual tool for framing our discussion of ``algebraizing'' and ``topologizing'': such 
actions can simply be seen as functorial. 

First, notice that Heyting algebras form a category, with the objects as Heyting algebras themselves, and morphisms given by 
Heyting algebra homomorphisms. Denote this category with $\mathbf{Heyt}$. For any two Heyting algebras $H$ and $K$, we will write 
$\Hom_{\mathbf{Heyt}}(H, K)$ to denote the collection of all homomorphisms between them (when the category we are working with is 
obvious, we will drop the subscript). Similarly, topological spaces form the objects of a category, which we will call $\mathbf{Top}$,
with morphisms given by continuous maps. 

Earlier, we noted that while Heyting algebra homomorphisms always induce continuous maps, continuous maps do not always induce 
Heyting algebra homomorphisms because of possible failures in preserving implication. As noted briefly before, however, they \textit{do}
induce homomorphisms of complete lattices, as completeness is a condition that can be phrased entirely in terms of joins and meets.
We will say that continuous maps thus induce \textit{frame} homomorphisms, and denote the relevant category as $\mathbf{Frm}$. 
It is easy to show (though we will not do so here), that a frame is always a complete Heyting algebra, so this simplifies our work 
considerably. We are now armed with enough information to construct our first functor.

\begin{proposition}
    The map $F : \mathbf{Top} \rightarrow \mathbf{Frm}^{op}$ is a contravariant functor.
\end{proposition}
\begin{proof}
    Let $F: \mathbf{Top} \rightarrow \mathbf{Frm}^{op}$ be the map that assigns to each topological space $X$ its corresponding frame $\mathcal{O}(X)$ of open sets and to each continuous map $f: Y \rightarrow X$ the frame homomorphism $F(f) = f^{-1}: \mathcal{O}(X) \rightarrow \mathcal{O}(Y)$, where $f^{-1}$ is the preimage map induced by $f$.
    If $F$ is contravariant, then for every for each topological space $X$, the identity map $\text{id}_X: X \rightarrow X$ should map to the identity homomorphism $\text{id}_{\mathcal{O}(X)}$ on the frame $\mathcal{O}(X)$.
    Indeed, the preimage map induced by $\text{id}_X$ is $\text{id}_{\mathcal{O}(X)}$, as for any open set $U \in \mathcal{O}(X)$, we have $\text{id}_X^{-1}(U) = U$. Thus, $F(\text{id}_X) = \text{id}_{\mathcal{O}(X)}$.
    For continuous maps $f: Z \rightarrow Y$ and $g: Y \rightarrow X$, we need to show that $F(g \circ f) = F(f) \circ F(g)$, where $F(g \circ f): \mathcal{O}(X) \rightarrow \mathcal{O}(Z)$ is the preimage map induced by the composition $g \circ f$.
    For any open set $U \in \mathcal{O}(X)$, we have
    \[
        F(g \circ f)(U) = (g \circ f)^{-1}(U) = f^{-1}(g^{-1}(U)) = F(f)(F(g)(U)).
    \]
    Hence, $F(g \circ f) = F(f) \circ F(g)$.
    Since $F$ preserves identity morphisms and reverses the composition of morphisms, it is a contravariant functor.
\end{proof}
We will now prove a result in the other direction.
\begin{proposition}
    The map $G : \mathbf{Frm} \rightarrow \mathbf{Top}^{op}$ is a contravariant functor.
\end{proposition}
\begin{proof}
    Recall that any frame $ H $ is also a complete Heyting algebra.
    Let $ F : \mathbf{Frm} \rightarrow \mathbf{Top}^{op} $ be the functor that assigns to each frame $ H $ a topological space $ X $ whose points correspond to the prime filters on $ H $, and to each frame homomorphism $ f : H \rightarrow K $, a continuous map $ F(f) = f^{-1} : Y \rightarrow X $, where $ Y $ is the set of prime filters on $ K $, and $ X $ is the set of prime filters on $ H $.
    To show that $ F $ is a contravariant functor, we verify the two functorial properties: preservation of identities and composition.
    First, consider the identity homomorphism $ \text{id}_{H} : H \rightarrow H $. For any prime filter $ P $ on $ H $, we have $ \text{id}^{-1}_{H}(P) = P $, so the induced map $ F(\text{id}_{H}) = \text{id}_{X} : X \rightarrow X $ is indeed the identity map on $ X $. Therefore, $ F(\text{id}_{H}) = \text{id}_{F(H)} $, verifying that $ F $ preserves identities.
    Next, let $ f : H \rightarrow K $ and $ g : K \rightarrow L $ be frame homomorphisms. We need to show that $ F(g \circ f) = F(f) \circ F(g) $. Consider any prime filter $ Q $ on $ L $. By the definition of $ F $, we have:
    \[
    F(g \circ f)(Q) = (g \circ f)^{-1}(Q)
    \]
    and 
    \[
    (F(f) \circ F(g))(Q) = F(f)(g^{-1}(Q)) = f^{-1}(g^{-1}(Q)) = (g \circ f)^{-1}(Q).
    \]
    Hence, $ F(g \circ f)(Q) = (F(f) \circ F(g))(Q) $ for all prime filters $ Q $ on $ L $, proving that $ F(g \circ f) = F(f) \circ F(g) $.
    Therefore, $ F $ is a contravariant functor from $ \mathbf{Frm} $ to $ \mathbf{Top}^{op} $.
\end{proof}
By these two results, it appears that $F$ and $G$ are ``complementary'' in a sense: they send objects and morphisms between reversed 
pairs of source and target categories. The notion of an \textit{adjunction} between functors formalizes this almost-ineffable complementarity.
\begin{definition}
    Let $F : \mathcal{C} \rightarrow \mathcal{D}$ and $G : \mathcal{D} \rightarrow \mathcal{C}$ be functors. We say that $F$ is 
    \textit{left adjoint} to $G$ when there is a natural isomorphism 
    \begin{align*}
        \Hom_\mathcal{D}(F(C), D) \simeq \Hom_\mathcal{C}(C, G(D))
    \end{align*} 
    for every object $C$ in $\mathcal{C}$ and $D$ in $\mathcal{D}$. If this is true, we may also say $G$ is \textit{right adjoint} to 
    $F$.
\end{definition}
\begin{proposition}
    The functors $F : \mathbf{Top} \rightarrow \mathbf{Frm}^{op}$ and $G : \mathbf{Frm} \rightarrow \mathbf{Top}^{op}$ constructed 
    previously are adjoint.
\end{proposition}
\begin{proof}
    We will show that $F$ is left adjoint to $G$ by constructing a natural isomorphism 
    \[
    \varphi_{X,H}: \text{Hom}_{\mathbf{Frm}^{op}}(F(X), H) \cong \text{Hom}_{\mathbf{Top}}(X, G(H)),
    \]
    natural in both $X \in \mathbf{Top}$ and $H \in \mathbf{Frm}$.
    Given a continuous map $f: X \rightarrow G(H)$, we can induce a frame homomorphism $\mathcal{O}(X) \rightarrow H$ by pulling back open sets along $f$, i.e. $\phi_f(U) = f^{-1}(U)$ for any $U \in \mathcal{O}(G(H))$. Since $H \cong \mathcal{O}(G(H))$, this gives an element of $\text{Hom}_{\mathbf{Frm}^{op}}(\mathcal{O}(X), H)$.

    Conversely, for a frame homomorphism $\phi: \mathcal{O}(X) \rightarrow H$, we can define a continuous map $g: X \rightarrow G(H)$ by mapping each point $x \in X$ to the prime filter in $H$ corresponding to $\phi^{-1}(x)$. This gives an element of $\text{Hom}_{\mathbf{Top}}(X, G(H))$.
    We must now show that these assignments are mutually inverse and natural.
    To check naturality, consider the following diagram
    \[
    \begin{tikzcd}
    X \ar[d, "g"'] \ar[r, "f"] & G(H) \ar[d, "G(h)"] \\
    Y \ar[r, "g^*"'] & G(K)
    \end{tikzcd}
    \]
    Here, $g^*: \mathcal{O}(Y) \rightarrow \mathcal{O}(X)$ is the frame homomorphism induced by the continuous map $g: Y \rightarrow X$. The commutativity of this diagram corresponds to the naturality condition for the isomorphism $\varphi_{X,H}$.
    To check the mutual inverse condition, we need to show that the following equation holds for all prime filters $Q$ on $K$
    \[
    (g \circ f)^{-1}(Q) = g^{-1}(f^{-1}(Q)),
    \]
    which follows from the definition of preimage and confirms that $\varphi_{X,H}$ respects the composition of morphisms.
    Therefore, we have established a natural isomorphism between the hom-sets, showing that $F$ is left adjoint to $G$.
\end{proof}
Note that this adjunction does not define an (anti)-\textit{equivalence} between these categories: it can be shown that not every frame 
generates a topological space with \textit{points}, and thus the adjunction above falls short of being a full duality of the categories 
involved (the appropriate duality is between the category of frames, and point-free generalizations of topological spaces, known as 
\textit{locales}). Nevertheless, our motivations being primarily logical, this adjunction suffices to summarize many of the relationships 
between topology and algebra we have been discussing throughout this section.

One interesting phenomenon here is the way the functors modify the morphisms between the objects of their categories. While a homomorphism 
of Heyting algebras always preserves implication, simply shifting towards viewing them as \textit{locales} changes the character of 
these maps, making it possible to treat this category as a target of functors from the category of topological spaces (continuous 
maps always induce locale homomorphisms, as we know). In a way, the categories of topological spaces and locales contain less 
\textit{logical} information in their morphisms than the more ``vanilla'' category of Heyting algebras does. In a way, this shows us 
the limits of topologizing. However, we \textit{did} prove some results that mitigate this -- maps that are open and continuous, for 
example, \textit{do} induce Heyting algebra homomorphisms. We can thus define restrictions of the functor to various classes of 
morphisms as needed using such results.

We have now developed all the machinery needed to start discussing \textit{dual} intuitionistic logic, which will henceforth be the 
primary focus of the paper. 

\section{Dual intuitionistic logic}
We will begin our discussion of dual intuitionistic logic nearly opposite to how we began the last section: namely, we begin with an 
exploration of duality or \textit{contravariance} in general. Thus, we will go from the abstract to the concrete. This is primarily 
because the abstractions we have developed in the previous section will make it possible to prove many of the results we went through 
concretely for Heyting algebras in a more concise and general way, saving us a lot of work, and sparing time to build the reader's 
intuition on dual intuitionistic logic -- which is genuinely much stranger than intuitionistic logic. Thematically similar results can be 
found in James \cite{james_1996_closed} and Shramko \cite{yaroslavshramko_2005_dual}, though neither source follows the progression we do 
(we instead mirror Johnstone's style once again here, though he does not write about this topic). 

Let $H$ be a Heyting algebra. Its order is given by $\leq$, and as was proven in section 1, for $a, b \in H$, $a \leq b$ if and only 
if $a \rightarrow b$ -- thus, the order on the Heyting algebra fully characterizes the behavior of its implication operation. Meets 
(and dually, joins) can thus be characterized in terms of implication as follows 
\begin{definition}
    For $a, b \in H$, $a \wedge b$ is defined as the largest $x \in H$ such that $a \rightarrow x$ and $b \rightarrow x$; in other 
    words, smallest $x$ such that $a \leq x$ and $b \leq x$. 
\end{definition}
This is exactly congruent with the definition of the meet being the \textit{greatest lower bound}, and it is easy to see that the 
symmetric formulation for joins also yields the \textit{lowest upper bound}. 

Let us construct a new structure $H^{op}$, with the same underlying set as $H$, but with the direction of the order given reversed:
that is, we induce the \textit{lattice} structure with the following holding for elements $a, b \in H$
\begin{align*}
    a \leq b \in H \iff a \geq b \in H^{op}.
\end{align*}
A few things are already clear about $H^{op}$. With its order, the top and bottom elements of $H$ are swapped, so we know that it is a 
bounded lattice. It is also distributive (see Grätzer for a proof). Because of the reversal of order, we may also show that joins 
and meets are \textit{reversed}, and that the notion of ``implication'' itself becomes dualized, opening the window into characterizing 
the difference 
\begin{proposition}
    Let $H$ be a Heyting algebra, and let $H^{op}$ denote the bounded distributive lattice with the opposite order. 
    Then for $a, b \in H$, $a \vee_H b = a \wedge_{H^{op}} b$, $a \wedge b = a \vee_{H^{op}} b$, and $a \rightarrow b = 
    a \leftarrow b$, where $\leftarrow$ is defined such that $a \leftarrow b \leq z \iff a \leq b \vee c$.  
\end{proposition}
\begin{proof}
Recall that $ H^{op} $ is defined as the lattice with the 
opposite order,
meaning for any 
$ a, b \in H $, 
we have that 
$a \wedge_{H^{op}} b = \sup\{a, b\} \quad \text{and} \quad a \vee_{H^{op}} b = \inf\{a, b\}$.
Since $ \sup $ in $ H^{op} $ corresponds to $ \inf $ in $ H $ and vice versa, we have
$a \vee_H b = a \wedge_{H^{op}} b \quad \text{and} \quad a \wedge_H b = a \vee_{H^{op}} b$.
Next, consider the implication operation $ \rightarrow $ in $ H $, which satisfies
$a \rightarrow b \text{ is the largest element } z \text{ such that } a \wedge_H z \leq b.$
In $ H^{op} $, the corresponding operation $ \leftarrow $ should satisfy
$a \leftarrow b \text{ is the largest element } z \text{ such that } a \wedge_{H^{op}} z \leq b \text{ in the opposite order.}$
Since $ \wedge_{H^{op}} $ corresponds to $ \vee_H $ in $ H $ and the order is reversed, this condition becomes
\[
a \leftarrow b \geq z \iff a \geq b \vee_H z \text{ in } H,
\]
which implies
$a \leftarrow b = a \rightarrow b.$
Therefore, the duality relationships hold as stated, completing the proof.
\end{proof}
Let us call structures obtained by dualizing Heyting algebras as above \textit{co-Heyting algebras}. We can also characterize a 
similar duality between homomorphisms.
\begin{proposition}
    $\varphi : L \rightarrow K$ is a homomorphism of Heyting algebras if and only if $\psi : L^{op} \rightarrow K^{op}$ is a homomorphism
    of co-Heyting algebras.
\end{proposition}
\begin{proof}
    Let $L$ and $K$ be Heyting algebras, and let $L^{op}$ and $K^{op}$ be their corresponding co-Heyting algebras.
    Let $\varphi : L \rightarrow K$ be a homomorphism of Heyting algebras. Then, for $a, b \in L$ $\varphi(a \vee_L b) = \varphi(a) \vee_K \varphi(b) =
    \varphi(a) \wedge_{K^{op}} \varphi(b)$, $\varphi(a \wedge_L b) = \varphi(a) \wedge_K \varphi(b) = \varphi(a) \vee_{K^{op}} \varphi(b)$,
    and $\varphi(a \rightarrow b) = \varphi(a) \rightarrow \varphi(b) = \varphi(a) \leftarrow \varphi(b)$.
    Define a map $\psi : L^{op} \rightarrow K^{op}$ with $\psi(a) = \varphi(a)$. Then, we see that we obtain the following, for 
    $a, b \in L^{op}$.
    \begin{align*}
        \psi(a \vee_{L^{op}} b) = \varphi(a) \vee_{K^{op}} \varphi(b) = \psi(a) \vee_{K^{op}} \psi(b) \\
        \psi(a \wedge_{L^{op}} b) = \varphi(a) \wedge_{K^{op}} \varphi(b) = \psi(a) \wedge_{K^{op}} \psi(b) \\
        \psi(a \leftarrow b) = \varphi(a) \leftarrow \varphi(b) = \psi(a) \leftarrow \psi(b)
    \end{align*}
    which shows that $\psi$ is a homomorphism of co-Heyting algebras. The argument to show that every co-Heyting algebra homomorphism
    induces a Heyting algebra homomorphism is symmetric, and thus we will end the proof here.
\end{proof}
This is actually enough to give us a formal duality between the categories of Heyting algebras and co-Heyting algebras.
\begin{proposition}
    Let $\mathbf{Heyt}$ denote the category of Heyting algebras, and $\mathbf{Heyt}^{op}$ denote the category of co-Heyting algebras,
    defined with the appropriate notions of homomorphism. Then $\mathbf{Heyt}$ is \textit{anti-isomorphic} to $\mathbf{Heyt}^{op}$.
\end{proposition}
\begin{proof}
    Define a contravariant functor $F : \mathbf{Heyt} \rightarrow \mathbf{Heyt}^{op}$ with $F(X) = X$ for Heyting algebras $X$ and 
    $F(f) = g : Y \rightarrow X$ for homomorphisms $f : X \rightarrow Y$ of Heyting algebras. Define another contravariant functor 
     $G : \mathbf{Heyt}^{op} \rightarrow \mathbf{Heyt}$ with $G(Y) = Y$ for co-Heyting algebras $Y$ and $G(g) = f : X \rightarrow Y$ for homomorphisms
     $g : Y \rightarrow X$ of co-Heyting algebras. Then, we see that for any Heyting algebra $X$, we have $G(F(X)) = G(X) = X$, thus 
     $G \circ F$ is the identity on objects of $\mathbf{Heyt}$ -- a symmetric argument can be used for objects in $\mathbf{Heyt}^{op}$.
     Similarly, for any homomorphism $f : X \rightarrow Y$ of Heyting algebras, we have $G(F(f)) = G(f : Y \rightarrow X) = f$, thus 
     $G \circ F$ is the identity on homomorphisms of $\mathbf{Heyt}$. A symmetric argument holds for objects and homomorphisms of 
     $\mathbf{Heyt}^{op}$, thus this completes the proof.
\end{proof}
A similar argument suffices to show an anti-isomorphic between \textit{any} category $\mathcal{C}$ and its opposite $\mathcal{C}^{op}$,
but we aimed to first demonstrate an incident of duality relevant to our paper that occurs in ``nature'', hence our detailed proof 
of this specific case. Philosophically speaking, we can see $\mathbf{Heyt}^{op}$ as consisting of the category whose objects are 
exactly the same as those of $\mathbf{Heyt}$, but will \textit{all} structural properties inverted. For example, let $f : A \rightarrow B$,
$g : B \rightarrow C$, and $h = g \circ f$ be morphisms in $\mathbf{Heyt}$. Then, we see that the following diagram 
\[\begin{tikzcd}
	A && B \\
	\\
	&& C
	\arrow["f", from=1-1, to=1-3]
	\arrow["h"', from=1-1, to=3-3]
	\arrow["g", from=1-3, to=3-3]
\end{tikzcd}\]
commutes on $C$. Under the image of the anti-isomorphism of categories (let's call the appropriate functor $O$ for now), the diagram 
becomes the following 
\[\begin{tikzcd}
	A && B \\
	\\
	&& C
	\arrow["{O(f)}"', from=1-3, to=1-1]
	\arrow["{O(h)}", from=3-3, to=1-1]
	\arrow["{O(g)}"', from=3-3, to=1-3]
\end{tikzcd}\]
which commutes on $A$ instead of $C$. This doesn't merely modify the \textit{internal} properties of the category (as we have seen 
from our construction), but also how it relates to other categories. Consider, for example, the functor $G : \mathbf{Frm} \rightarrow \mathbf{Top}^{op}$
we constructed in the previous section. It is not difficult to show that this functor is naturally isomorphic to some other 
$G' : \mathbf{Frm}^{op} \rightarrow \mathbf{Top}$ defined with similar (only dual) assignments -- thus, we are able to able to easily
\textit{dualize} a functor's domain to construct an equivalent functor, meaning that dual categories \textit{share} functors, only 
inverted (e.g $G$ is contravariant, whereas $G'$ could be seen as covariant).

Speaking more concretely, this means that if we take categories for which we have a duality, such as Heyting algebras and topological 
spaces, for example, we can obtain a duality between \textit{co}-Heyting algebras and topological spaces as well, and thus topologize 
the logic of co-Heyting algebras in the exact same way that we proceeded for Heyting algebras. Because dualization translates 
\textit{all} categorical notions over, as long as we pick morphisms that preserve desired logical properties (in this case, Heyting 
and co-Heyting algebra homomorphisms), we may topologize logical notions about co-Heyting algebras \textit{from} the corresponding 
notions about Heyting algebras without fear. Our abstractions have thus allowed us to engage in concrete constructions without worry 
about whether they \textit{actually} reflect logical properties we want to preserve. Let us use this knowledge to prove some results 
about co-Heyting algebras. Our proofs will proceed fairly concretely and topologically, to maximize logical insight.

\begin{proposition}
    Let $X$ be a topological space. Then its lattice of closed sets $\mathcal{C}(X)$ is a co-Heyting algebra.
\end{proposition}
\begin{proof}
    Recall that for any topological space $X$, its lattice of open sets $\mathcal{O}(X)$ is a Heyting algebra. Recall also 
    that a set $A \subseteq X$ is \textit{closed} iff $X \backslash A$ is open -- we will denote that a set $A$ is closed with 
    $A \in \mathcal{C}(X)$. Then, $X, \emptyset \in \mathcal{C}(X)$, as they are open and each other's complements. From this it 
    follows that $\mathcal{C}(X)$ is closed under arbitrary intersections and finite unions, thus it is a lattice.

    We may induce an order on $\mathcal{C}(X)$ by stipulating that $A \leq_{\mathcal{C}(X)} B$ for $A, B \subseteq$ if and only 
    if $X \backslash A \leq_{\mathcal{O}(X)} X \backslash B$ (i.e. if and only if $X \backslash A \subseteq X \backslash B$). 
    From this, we obtain that $\mathcal{C}(X)$ is bounded, with $\emptyset$ as the top element, and $X$ as the bottom element. 
    Thus, for $A, B \in \mathcal{C}(X)$, we have $A \wedge B = A \cup B$, and $A \vee B = A \cap B$. 

    It remains to show that the co-Heyting implication (subtraction) $A \leftarrow B$ exists in $\mathcal{C}(X)$. To this end, 
    let us define
    \[
    A \leftarrow B := \mathrm{cl}({A \cap B^c}),
    \]
    where $\mathrm{cl}({A \cap B^c})$ denotes the closure of the intersection $A \cap B^c$ in $X$. By construction, 
    $A \leftarrow B$ is a closed set, and it is the smallest closed set containing the intersection of $A$ and the complement of $B$, which 
    makes it a subtractive operation, and precisely dual to Heyting implication (as that is the interior of a similar union, we simply 
    dualize both operations).
    Therefore, $\mathcal{C}(X)$ is indeed a co-Heyting algebra, completing the proof.
\end{proof}
A converse result, establishing that any co-Heyting algebra can be embedded into the lattice of closed sets of a topological space,
can also be proven, see James \cite{james_1996_closed} for a proof. Given the similarity to the corresponding proof for Heyting algebras in the first section (and the possibility 
of using a simple functorial proof), we will not spend time on this at the moment (although the next section contains a discussion on it,
with relevance to modal logic). Instead, armed with the comfort that algebro-topological dualities 
can be shown to hold with easy categorical arguments, we will dive into developing the basic concepts of dual intuitionistic logic 
centered around the co-implication operation defined in the proof above.

The utter strangeness of this dual intuitionistic logic should be apparent to the reader from the strategies used in the proof above.
For closed sets $A$ and $B$ (corresponding to elements of a co-Heyting algebra, thus to propositions), the co-implication $A \leftarrow B$
is defined as the closure of the intersection of $A$ and $B^c$: in other words, the closure of the set of all points in $A$ and 
\textit{not} in $B$. We may thus conceptualize this as measuring the degree to which one can affirm $A$ while affirming all that is 
not $B$, hence the use of ``subtraction'' to describe this operation.

The subtractive character of co-Heyting implication leads to an interesting concept of negation, adjoint to Heyting 
pseudocomplementation.
\begin{definition}
    Let $H$ be a co-Heyting algebra. Then, for $a \in H$, the \textit{co-Heyting negation} of $a$ is defined as 
    the element ${\sim} a = 1 \leftarrow a$. 
\end{definition}
What is the sense in this definition? Well, supposing that $X$ is the topological space that $H$ embeds into, we have that 
$a \mapsto A \in \mathcal{C}(X)$, and $1 \mapsto X$. Further topologizing (and somewhat abusing notation), we thus see that 
$1 \leftarrow a = \mathrm{cl}(X \cap A^c)$, or the closure of all points in $X$ that are not in $A$. This is why Lawvere called 
this operation \textit{non}-$a$, for a proposition $a$: the construction is much more ``liberal'' than that of the Heyting 
pseudocomplement (see \cite{lawvere_1991_intrinsic} for more discussion of this). It is easily seen that because $\leftarrow$ is adjoint to $\rightarrow$ in Heyting algebras, we can \textit{dualize}
the relationship pseudocomplementation shares to negation to obtain that ${\sim} a \leq x \textrm{ iff } 1 = a \vee x$ for any $a$ in a 
co-Heyting algebra $H$. In other words, ${\sim} a$ is the least $b \in H$ with $a \vee b = 1$, or topologically speaking, the smallest
closed set whose union with the set corresponding to $a$ covers the whole space $H$ embeds into. This definition has a number of 
interesting consequences, which we will characterize now (one trivial one is that ${\sim}{\sim}a \leq a$).

First, observe that unlike in intuitionistic logic, we can retain a version of the conjunctive De Morgan law in dual intuitionistic 
logic, although the disjunctive version does not always hold. We will call this the \textit{dual de Morgan law}. 
\begin{proposition}
    Let $H$ be a co-Heyting algebra. Then for $a, b \in H$, we have ${\sim}(a \wedge b) = {\sim}a \vee {\sim}b$.
\end{proposition}
\begin{proof}
    We handle this proof topologically, using the duality we have constructed. Let $a \mapsto A$ and $b \mapsto B$. Then 
    ${\sim}(a \wedge b)$ corresponds to the closed set $\mathrm{cl}((A \cap B)^c)$. Suppose $x \in \mathrm{cl}((A \cap B)^c)$. 
    Then either $x \in \mathrm{cl}(A^c$), or $x \in \mathrm{cl}(B^c)$, or both, thus $x \in \mathrm{cl}(A^c) \cup \mathrm{cl}(B^c)$.
    On the other hand, suppose $x \in \mathrm{cl}(A^c) \cup \mathrm{cl}(B^c)$. Then $x \not \in \mathrm{int}(A)$ or $x \not \in 
    \mathrm{int}(B)$. Thus $x \not \in \mathrm{int}(A \cap B)$, and from this we may conclude that it is in the complement, that is
    $x \in \mathrm{cl}((A \cap B)^c)$. Since we have shown that $x \in \mathrm{cl}((A \cap B)^c)$ if and only if $x \in \mathrm{cl}(A^c) 
    \cup \mathrm{cl}(B^c)$, by duality, we have completed the proof of the corresponding lattice-theoretic equality.
\end{proof}

This understanding of negation also allows us to construct an interesting operator, known as the co-Heyting boundary operator. We introduce 
the formalism now.
\begin{definition}
    Let $H$ be a co-Heyting algebra, and let $a \in H$. Then, define the \textit{co-Heyting boundary operator} for $a$ with 
    $\partial a := a \wedge {\sim} a$.
\end{definition}
In intuitionistic logic, such an operator would always yield empty output. However, this is not the case in dual intuitionistic logic!
\begin{proposition}
    Let $H$ be a co-Heyting algebra, and let $a \in H$. Then $\partial a$ is not necessarily empty.
\end{proposition}
\begin{proof}
    Let $X$ be a topological space, and let $\mathcal{C}(X)$ be its lattice of closed sets. Suppose that $\mathcal{C}(X)$ is the closed
    set lattice that $H$ embeds into, and thus suppose $a \mapsto A$, where $A \in \mathcal{C}(X)$. Then, we see that we may write 
    \begin{align*}
        \partial a &= a \wedge {\sim} a \\
                   &= A \cap \mathrm{cl}(X \cap A^c) \\
                   &= A \cap \mathrm{cl}(A^c).
    \end{align*}
    While $A \cap A^c$ is necessarily empty, the closure $\mathrm{cl}(A^c)$ may contain points in $A$ (there are many examples of topological 
    spaces where this is the case). Thus, it follows that $\partial a$ is not necessarily empty, which completes the proof.
\end{proof}
So, dual intuitionistic logic is \textit{paraconsistent}, and thus the law of noncontradiction need not hold in a co-Heyting algebra 
$H$. However, as we will see, the law of excluded middle \textit{will} hold.
\begin{proposition}
    Let $H$ be a co-Heyting algebra, and let $a \in H$. Then, the union of closed sets corresponding to $a \vee {\sim} a$ always 
    covers the whole space associated with $H$. 
\end{proposition}
\begin{proof}
    Let $X$ and $\mathcal{C}(X)$ be the space associated with $H$ and its closed set lattice respectively. Suppose $a$ maps to a 
    closed set $A \in \mathcal{C}(X)$. By definition, we have that ${\sim} a = \mathrm{cl}(X \cap A^c)$, which we will refer to as 
    ${\sim} A$ for short. We are considering the set $A \cup \mathrm{cl}(X \cap A^c)$, which can also be just written as $A \cup 
    \mathrm{cl}(A^c)$. Let $x \in X$. Suppose $x \in A$. Then, $x \in A \cup \mathrm{cl}(A^c)$ by the properties of the union. On the 
    other hand, if $x \not \in A$, then we have $x \in \mathrm{cl}(A^{c})$, and $x \in A \cup \mathrm{cl}(A^c)$ once again by the properties of 
    unions. We can thus assign any arbitrary point of the space $x$ to either term of the union, which means that 
    $A \cup \mathrm{cl}(A^c) = X$, completing the proof.
\end{proof}
Now that we have proven some of the basic (strange!) properties of dual intuitionistic logic, it is worth to take
a second to pause and think about how we actually ought to interpret this logic. Symbol-pushing can only get us so far. If intuitionistic logic is the logic reasoning ``constructively''
about how proofs may be constructed from their terms (which is formalized by the topology developed in the first section), how should 
we think about dual intuitionitic logic? One perspective, which we will take in this paper (due to the view of Yaroslav Shramko), is 
that dual intuitionistic logic is the logic of \textit{falsifiability}. In order to understand why this is the case, we will take a 
small detour into the world of abstract algebraic logic to explain the idea of a \textit{consequence relation}. The following discussion 
will be informal in character and is included mainly for intuition: a more rigorous treatment can be found in Shramko or Font \cite{yaroslavshramko_2005_dual} \cite{josepmariafont_2016_abstract}.
\begin{definition}
    A \textit{consequence relation} on a set $A$ is a relation $\vdash \subseteq \mathcal{P}(A) \times A$ such that 
    for all $X \cup Y \cup \{x\} \subseteq A$, we have the following properties 
    \begin{enumerate}
        \item If $x \in X$, then $X \vdash x$,
        \item If $x \in X$ and $X \subseteq Y$, then $Y \vdash x$,
        \item If $X \vdash y$ for all $y \in Y$ and $Y \vdash x$, then $X \vdash x$. 
    \end{enumerate}
\end{definition}
Now suppose we have a consequence relation $\vdash$ on a set $A$ that is \textit{substitution-invariant}: that is, for any 
permutation $\sigma : A \rightarrow A$ that preserves any syntactic structure we choose to put onto $A$, we have for $x, y \in A$ 
that $x \vdash y$ if and only if $\sigma(x) \vdash \sigma(y)$. Since we have assumed that $\sigma$ respects any structure we deem 
relevant on $A$, a consequence relation that is substitution-invariant is essentially one where the consequentiality is preserved 
\textit{across} structurally equivalent pairs of terms. Thus, the idea of a substitution-invariant consequence relation essentially 
captures the idea of a \textit{logic}, where the form of a proposition matters more than its content -- that is, equivalent forms, even 
if not literally equal, should entail and \textit{be entailed} by the same formulae. 

Given a substitution-invariant consequence relation $\vdash$ on a set $A$, it is possible to define \textit{rules of inference} on 
elements treated as formulae in $A$. One way to do this is to treat $A$ as the \textit{free algebra} generated by atomic 
symbols $x_i$ indexed with $i \in \mathbb{N}$, conjoined with operator symbols like $\vee$, $\wedge$, $\neg$, etc..., subject to 
well-formedness rules. Then, a \textit{rule of inference} is just a schema of the form $a \vdash b$, where $a$ and $b$ are constructed 
as specified above. A major mathematical acheivement of the early 20th century was the codification of various logics into sets of such 
rules of inference. This enabled a combinatorial and essentially mechanical study of what had been seen as semi-ineffable ``laws of thought'',
completing a project initiated by George Boole. It also dethroned the primacy of a certain mystique that some people attached to various 
methods of reasoning: by specifying a consequence relation and a specific set of inference rules, in principle, it is not difficult 
to construct a logic that expresses the notion of consequence in any arbitrary way.

Intuitionistic logic (and constructive logic more generally) owes its intellectual lineage to the \textit{intuitionistic} point of 
view on the foundations of mathematics, dating back to work of Dutch mathematician L.E.J Brouwer the early 20th century. Intuitionism 
endorsed the central claim that mathematical objects were not objectively existent or ``out there'' in any Platonic sense, but rather 
constructions of the human mind, and directly linked to our capacity to reason. The key apparatus for engaging in such constructions,
of course, is that of the \textit{mathematical proof}. A fully constructive proof is one where is a ``construction'' proving each part 
of the proposition to be proven. Having a constructive proof is the criterion for intuitionistic truth, thus a tautology is a statement 
which can \textit{always} be constructively proven.

From this point of view, it is easy to see why Brouwer rejected the law of excluded middle. By the criteria for intuitionistic truth,
the sentence $p \vee \neg p$ is true if and only if there is \textit{always} a constructive proof available for $p$ or $\neg p$: however,
Brouwer did not see this to be philosophically possible, because it was of course always possible that neither parameter could be 
proven constructively (intuitively illustrated by our discussion in the first section). Could we construct a consequence relation 
capturing this idea? Without getting into too much detail, this is precisely what Brouwer's student Arend Heyting did. Despite Brouwer 
being opposed to the formalization of intuitionism (indeed, he did not view it as philosophically desirable), by presenting a syntactic 
and semantic characterization of a consequence relation that \textit{captured} the idea behind intuitionistic logic, Heyting did essentially this.
Let us call the consequence relation associated with intuitionistic logic $\vdash_I$ for the purposes of our discussion.

Suppose that we have $\Gamma \vdash_I \varphi$. Then, we can say that by a series of purely constructive manipulations (a la Brouwer, Heyting and 
Kolmogorov): in a sense, the assumptions in $\Gamma$ constructively \textit{prove} $\varphi$. How do we dualize this notion? The dual 
of the notion of provability is refutability, and so supposing $\vdash_D$ is the consequence relation for dual intuitionistic logic, 
let us say that $\Gamma \vdash_D \varphi$ if and only if $\varphi$ is \textit{not refutable} from $\Gamma$. In this intuitive picture,
it is entirely possible to see that we can have both $\Gamma \vdash_D \varphi$ and $\Gamma \vdash_D \neg \varphi$, where neither 
$\varphi$ nor $\neg \varphi$ is necessarily refutable. All of the inference schemes of intuitionistic logic can be used to obtain dual 
intuitionistic logic, and co-Heyting algebras -- the very objects for which we meticulously developed the structure theory above -- 
give a complete semantics for the resulting logic.

With this discussion in mind, closed sets in a topological space may be interpreted as propositions that are either \textit{refutable}
or \textit{not refutable} by the axioms of the theory that the space interprets through the co-Heyting algebra given by its lattice 
of closed sets. The boundary operator $\partial a$ in particular captures how ambiguous the refutability of the proposition corresponding 
to $a$ is -- indeed, the below proposition shows that it enjoys a number of highly nontrivial properties. See Reyes and Zolfaghari \cite{reyes_1996_biheyting}
for a similar proof.

\begin{proposition}
    Let $H$ be a co-Heyting algebra, and let $a, b \in H$. Then, the following properties are always true of the co-Heyting 
    boundary operator.
    \begin{enumerate}
        \item $\partial(a \wedge b) = (\delta a \wedge b) \vee (a \wedge \delta b)$
        \item $\partial a \vee \partial b = \partial(a \vee b) \vee \partial(a \wedge b)$
        \item If $x = \partial a$, then $\partial x = x$
        \item $a = {\sim}{\sim}a \vee \partial a$
    \end{enumerate}
\end{proposition}
\begin{proof}
    We prove each part by considering the lattice of closed sets that $H$ embeds into. Recall that $\partial a := 
    a \wedge {\sim} a$, and that ${\sim}a \mapsto \mathrm{cl}(A^c)$, where $a \mapsto A$.
    \begin{enumerate}
        \item 
        Notice that $\partial(a \wedge b) = \partial((a \wedge b) \wedge {\sim}(a \wedge b)) = \partial((a \wedge b) \wedge ({\sim}a 
        \vee {\sim}b))$, using the dual De Morgan law. Then, since $H$ is distributive, we may write $\partial((a \wedge b) \wedge ({\sim}a 
        \vee {\sim}b)) = \partial((a \wedge {\sim}a \wedge b) \vee (a \wedge b \wedge {\sim}b))$. From here, notice that we may use 
        the fact that the boundary operator distributes over joins to write $\partial((a \wedge {\sim}a \wedge b) \vee 
        (a \wedge b \wedge {\sim}b)) = \partial(a \wedge {\sim}a \wedge b) \vee \partial(a \wedge b \wedge {\sim}b) = 
        (\partial a \wedge b) \vee (a \wedge \partial b)$, completing the proof of this part. 
        \item
        The proof of this part is most naturally approached by translating to the language of closed sets. We have that $\partial a \vee \partial b \mapsto 
        (A \cap \mathrm{cl}(A^c)) \cup (B \cap \mathrm{cl}(B^c))$ if $a \mapsto A$ and $b \mapsto B$. Suppose $x \in 
        (A \cap \mathrm{cl}(A^c)) \cup (B \cap \mathrm{cl}(B^c))$. Then, $x$ is either in the boundary of $A$ and $A^c$, in the 
        boundary of $B$ and $B^c$, or both. Thus it is true that either $x \in (A \cup B) \cap \mathrm{cl}((A \cup B)^c)$ (i.e 
        either $x$ is in $A \cup B$, and $x$ is in the closure of its complement), or $x \in (A \cap B) \cap \mathrm{cl}((A \cap B)^c)$
        (i.e in both of $A$ and $B$, and in the closure of the boundary of their intersection). Thus $x \in ((A \cup B) \cap \mathrm{cl}((A \cup B)^c))
        \cup ((A \cap B) \cap \mathrm{cl}((A \cap B)^c))$. 

        On the other hand, suppose that $x \in ((A \cup B) \cap \mathrm{cl}((A \cup B)^c)) \cup ((A \cap B) \cap \mathrm{cl}((A 
        \cap B)^c))$. Then either $x \in (A \cup B) \cap \mathrm{cl}((A \cup B)^c)$, in which case $x$ must be either a boundary 
        point of $A$, or a boundary point of $B$, that is, $x \in (A \cap \mathrm{cl}(A^c)) \cup (B \cap \mathrm{cl}(B^c))$. Then,
        if $x \in ((A \cap B) \cap \mathrm{cl}((A \cap B)^c))$, we have the same, and thus once again we may derive $x \in 
        (A \cap \mathrm{cl}(A^c)) \cup (B \cap \mathrm{cl}(B^c))$. Since self-union is idempotent, we conclude that $x \in 
        (A \cap \mathrm{cl}(A^c)) \cup (B \cap \mathrm{cl}(B^c))$. 

        Since we have shown that $x \in ((A \cup B) \cap \mathrm{cl}((A \cup B)^c))
        \cup ((A \cap B) \cap \mathrm{cl}((A \cap B)^c))$ if and only if $x \in x \in 
        (A \cap \mathrm{cl}(A^c)) \cup (B \cap \mathrm{cl}(B^c))$, we have completed the proof that the corresponding lattice-theoretic 
        statements are equal by duality. This concludes the proof of this part.
        \item
        Suppose that $x = \partial a$. Then $x = a \wedge {\sim}a$, thus $\partial x = (a \wedge {\sim}a) \wedge {\sim}(a \wedge {\sim}a)$.
        By the dual de Morgan law, we have ${\sim}(a \wedge {\sim}a) = ({\sim}a \vee {\sim}a)$. Thus $\partial x = (a \wedge {\sim}a) \wedge 
        ({\sim}a \vee {\sim}a)$. Since $(a \wedge {\sim}a) \leq ({\sim}a \vee {\sim}a)$, it follows by the definition of meet that 
        $\partial x = (a \wedge {\sim}a) = x$, completing the proof of this part.
        \item
        Notice first that ${\sim}{\sim}a \vee \partial a = {\sim}{\sim}a \vee (a \wedge {\sim}a)$. We may apply the distributive law 
        to obtain ${\sim}{\sim}a \vee (a \wedge {\sim}a) = ({\sim}{\sim}a \vee a) \wedge ({\sim}{\sim}a \vee a)$. Since ${\sim}{\sim}a
        \leq a$, we have $({\sim}{\sim}a \vee a) \wedge ({\sim}{\sim}a \vee a) = a \wedge a = a$, showing that ${\sim}{\sim}a \vee \partial a 
        = a$, and completing the proof of this part.
    \end{enumerate}
    Since we have proven each part, we have shown that these properties characterize the co-Heyting boundary operator.
\end{proof}
These properties of co-Heyting boundaries as elements of co-Heyting algebras allow us to manipulate them algebraically, generalizing 
the concept of a boundary from topology to our logical context. For Heyting algebras, the most powerful bookkeeping tools we had to 
deal with sets of propositions (in terms of maximal consistency or inconsistency) were filters and ideals. We will now show that filters 
and ideals in co-Heyting algebras are closely related to co-Heyting boundaries, and that the relationship is crucial to understanding 
the nature of ``boundary decompositions'' of the form $a = {\sim}{\sim}a \vee \partial a$. The first result we will prove is a surprising 
relationship that Boolean algebras share with co-Heyting algebras.

\begin{proposition}
    Let $H$ be a co-Heyting algebra. Then $H$ is a Boolean algebra if and only if $\partial a = 0_H$ for all $a \in H$.
\end{proposition}
\begin{proof}
    Suppose that $\partial a = 0_H$ for all $a \in H$. Note that since the co-Heyting negation of an element always exists, 
    we have that $a \wedge {\sim}a = 0_H$ for all $a \in H$, meaning that ${\sim}a$ satisfies one property of the Boolean complement.
    To see that $a \vee {\sim}a = 1_H$, recall that the law of excluded middle always holds in a co-Heyting algebra. Thus ${\sim}a$
    satisfies all the properties of the Boolean complement of $a$, and since $H$ is a bounded, distributive, fully complemented lattice,
    we conclude that $H$ is a Boolean algebra, with $a \rightarrow b$ given by ${\sim}a \wedge b$.

    On the other hand, suppose that $H$ is a Boolean algebra. Note that $H$ carries the structure of a co-Heyting algebra if we 
    set $a \leftarrow b := a \wedge \neg b$. It is then easy to see that ${\sim}a = 1_H \leftarrow a = 1_H \wedge \neg a = \neg a$.
    Thus ${\sim} a = \neg a$, and since complementation in a Boolean algebra is unique, it follows that $\partial a = a \wedge \neg a
    = 0_H$ for all $a \in H$, completing the proof.
\end{proof}
Note that this result is essentially the same as saying that if all boundaries are trivial in a topology on a set $X$, then all subsets of $X$
must be clopen. Thus, the above is a criterion for the topological space corresponding to $H$ being \textit{totally disconnected}. Indeed, there 
is an intimate relationship between the ``triviality'' of the negation operation on a bi-Heyting (both Heyting and co-Heyting) algebra (where the most trivial operation is 
simple Boolean negation) and the connectedness of the corresponding spectral space. 

Our characterization of dual intuitionistic logic has thus far been primarily algebraic and topological. While every topological space induces 
both a Heyting and co-Heyting algebra (which coincide if and only if the lattice of subsets gives a Boolean algebra), this abstraction does not 
make the nature of the \textit{relationship} between these structures clear. 

\section{Kripke semantics}

In this section, we fix a Heyting algebra $H$ and the corresponding co-Heyting algebra $H^{op}$. Let $X = \mathrm{Spec}(H)$ be the set of prime 
filters on $H$, and thus the topological space for which $H$ is the lattice of open sets. Note that $\mathrm{Spec}(H^{op}) = \mathrm{Spec}(H)
 = X$. Since $H$ and $H^{op}$ share this spectral space, it seems natural to assume that perhaps a logical characterization of this spectral space
could yield a logical characterization of how the intuitionistic theory given by $H$ interacts with the dual intuitionistic theory given by 
$H^{op}$. As it turns out, with the tools of \textit{modal logic}, we may interpret $\mathrm{Spec}(H)$ through Kripke semantics as a space of 
\textit{possible worlds} where propositions are either satisfied or not satisfied. First, we give some essential definitions, following the 
exposition on modal logic from Fagin et. al. \cite{fagin_2011_reasoning}. 

\begin{definition}
    A modal language $\mathcal{L}$ is any language given by for which formulae $\varphi$ are given by the BNF grammar $\langle \varphi
    \rangle ::= p \mid \bot \mid \neg \varphi \mid \varphi \wedge \psi \mid \varphi \vee \psi \mid \varphi \rightarrow \psi \mid 
    \Diamond \varphi \mid \Box \varphi$, where $p \in \mathcal{L}$ is an atomic formula.
\end{definition}
\begin{definition}
    A Kripke frame is an ordered pair $(W, R)$, where $W$ is a set of possible worlds, and $R \subseteq W \times W$ is a binary relation on $W$.
    Elements $w \in W$ may simply be denoted as worlds, and the relation $R$ is called an accessibility relation. 
\end{definition}
\begin{definition}
    A Kripke model is a triple $(W, R, \models)$, where $(W, R)$ is a Kripke frame, and $\models$ is a relation between worlds $w \in W$ and 
    formulae $\varphi$ and $\psi$ of a modal language $\mathcal{L}$ such that 
    \begin{enumerate}
        \item $w \models \neg \varphi$ if and only if $w \not \models \varphi$,
        \item $w \models \varphi \rightarrow \psi$ if and only if $w \not \models \varphi$ or $w \models \psi$,
        \item $w \models \Box \varphi$ if and only if $u \models \varphi$ for all $u \in W$ such that $wRu$.
    \end{enumerate}
\end{definition}
We may read $w \models \varphi$ as ``$w$ models $\varphi$'', ``$w$ satisfies $\varphi$'' or ``$w$ forces $\varphi$'' (forcing is to 
be understood loosely by analogy to the idea in set theory). The action of the satisfaction relation on propositions and worlds is determined 
by the inductive application of the above definition on atomic formulae $p \in \mathcal{L}$, for some fixed modal language $\mathcal{L}$. We 
say that a formula $\varphi \in \mathcal{L}$ is valid in a Kripke model $(W, R, \models)$ if $w \models \varphi$ for every $w \in W$, and valid 
in the corresponding Kripke frame $(W, R)$ if $w \models \varphi$ for every $w \in W$ and every possible satisfaction relation $\models$. 
The validity of a formula $\varphi$ on a class $C$ of frames or models is attained if $\varphi$ is valid in every member of $C$. We denote by 
$\mathrm{Thm}(C)$ the set of all valid formulas in $C$, and by $\mathrm{Mod}(X)$ the class of all Kripke frames validating all formulae from some 
set $X \subseteq \mathcal{L}$. We say that $\mathcal{L}$ is sound with respect to a class of frames $C$ if $\mathcal{L} \subseteq \mathrm{Thm}(C)$, and 
that a class of frames $C$ is complete with respect to a modal language $\mathcal{L}$ if $\mathrm{Thm}(C) \subseteq \mathcal{L}$. 
For atomic formulae $p \in \mathcal{L}$, the truth value is assigned directly via a set of worlds $W_p \subseteq W$ such that $w 
\models p$ if and only if $w \in W_p$. This assignment provides the base case for determining the satisfaction of more complex formulae
inductively.

The operators $\Box$ and $\Diamond$ gain their meaning in modal logic through the accessibility relation $R$ in a Kripke frame. Specifically,
for any formula $\varphi$, $\Box \varphi$ represents a statement about the universality of $\varphi$ across all worlds accessible from a 
given world $w$, while $\Diamond \varphi$ encapsulates the possibility of $\varphi$ being true in at least one such world. Formally, 
these interpretations are given by
\[
w \models \Box \varphi \iff \forall u \in W, (wRu \implies u \models \varphi),
\]
\[
w \models \Diamond \varphi \iff \exists u \in W, (wRu \wedge u \models \varphi).
\]
Thus, the accessibility relation $R$ serves as the geometric backbone of modal logic, encoding how truths in one world influence or constrain
truths in other worlds. The logical structure imposed by $\Box$ and $\Diamond$ mirrors the properties of $R$. For instance, reflexivity 
of $R$ ensures that $\Box \varphi \implies \varphi$, and transitivity of $R$ underpins $\Box \varphi \implies \Box \Box \varphi$, 
which are central axioms of the modal system $\mathsf{S4}$.
This connection allows us to classify modal logics by the structural constraints placed on $R$. 

How does the structure of $R$ influence the logic we are dealing with? To answer this, we must first investigate the family of \textit{normal} modal logics. 
\begin{definition}
    A normal modal logic is a set of modal formulae $\mathcal{L}$ such that $\mathcal{L}$ contains all propositional tautologies, and for any 
    $\varphi, \psi \in \mathcal{L}$, all instances of the Kripke schema $\Box (\varphi \rightarrow \psi) \rightarrow \Box \varphi \rightarrow \Box \psi$.
    In addition, $\mathcal{L}$ must satisfy the necessitation rule, i.e if $\varphi$ is a theorem of $\mathcal{L}$, then so is $\Box \varphi$. 
\end{definition}
Normal modal logics thus extend classical logic with distribution of the $\Box$ operator over implication, and the necessitation rule. Together,
these mean that the normal modal logics all satisfy a sort of ``modal modus ponens''. The philosophical significance of this becomes clear once 
one reads $\Box \varphi$ as ``it is necessarily the case that $\varphi$''. This motivates the definition of $\Diamond \varphi = \neg \Box \neg \varphi$,
i.e if it is not necessarily the case that $\neg \varphi$, then $\varphi$ is possible. While this characterization means that normal modal logics 
already possess some ``nice'' degree of structure, we have not put any constraint on the accessibility relations for the Kripke frames corresponding 
to the class of normal modal logics. However, it turns out that changing the accessibility relations gives us an interesting way to \textit{parametrize}
normal modal logics, as seen by the following definition.
\begin{definition}
Let $\mathcal{F} = (W, R)$ be a Kripke frame, where $W$ is a set of worlds and $R \subseteq W \times W$ is an accessibility relation. 
Let $\mathcal{L}$ be a normal modal language, and suppose that $\mathcal{L}$ is sound with respect to $\mathcal{F}$, and that $\mathcal{F}$ 
is complete with respect to $\mathcal{L}$ for some derivability relation $\models$. Thus $\mathcal{F}$ models a normal modal logic.
We iteratively define a hierarchy of normal modal logics as follows.
\begin{enumerate}
    \item $\mathsf{K}$: No restrictions are placed on $R$. Thus, we have just $\Box(\varphi \rightarrow \psi)$, for $\varphi, \psi \in \mathcal{L}$.
    \item $\mathsf{T}$: The relation $R$ is reflexive, i.e. for all $w \in W$, $wRw$. Equivalently, we have $\Box \varphi \rightarrow \varphi$ for any 
    $\varphi \in \mathcal{L}$ in addition to any necessary constraints to $\mathsf{K}$. 
    \item $\mathsf{S4}$: The relation $R$ is reflexive and transitive, i.e. for all $w, v, u \in W$,
    \[
    wRw \quad \text{and} \quad (wRv \text{ and } vRu) \rightarrow wRu
    \]
    holds. Equivalently, we have $\Box \varphi \rightarrow \Box \Box \varphi$ for any $\varphi \in \mathcal{L}$ in addition to any necessary constraints 
    to $\mathsf{T}$. 
    \item $\mathsf{S5}$: The relation $R$ is an equivalence relation, i.e. $R$ is reflexive, transitive, and symmetric, i.e. for all $w, v \in W$,
    \[
    wRv \rightarrow vRw
    \]
    holds. Equivalently, we have $\varphi \rightarrow \Box \varphi$ for any $\varphi \in \mathcal{L}$ in addition to any necessary constraints to $\mathsf{S4}$. 
\end{enumerate}
\end{definition}
It is easy to see that anything that can be derived in $K$ can be derived at more restricted stages of the hierarchy, and more generally, that 
every theorem of a system with a less restricted accessibility relations is a theorem of a system with a more restricticted accessibility relation.
The first step we will take towards translating the logic of topological spaces into modal language is by identifying whether topological spaces 
can model a normal modal logic, and if so, what kind of accessibility relation we can assign. The following theorem will completely answer these 
questions. While the result is standard, we have cast the proof in the language of Heyting algebras and co-Heyting algebras, which the author 
believes is novel.

\begin{theorem}
    Let $\mathcal{L}$ be a normal modal language satisfying the $\mathsf{S4}$ axioms, let $\mathcal{T}$ be the class of topological spaces,
    and let $X \in \mathcal{T}$ be arbitrary. For any formula $\varphi \in \mathcal{L}$, let $\mathcal{O}_\varphi \subseteq X$ denote
    the set of points where $\varphi$ is satisfied, which we interpret as an open set if $\varphi$ is of the form $\Box \psi$, or 
    as a closed set if $\varphi$ is of the form $\Diamond \psi$. Then
    \[
    \Box \varphi = \mathrm{int}(\mathcal{O}_\varphi), \quad \Diamond \varphi = \mathrm{cl}(\mathcal{O}_\varphi).
    \]
    If $\varphi$ is derivable in $\mathsf{S4}$, then for every topological space in $\mathcal{T}$, $\mathcal{O}_\varphi$ satisfies the
    interpretations of $\Box \varphi$ and $\Diamond \varphi$. Conversely, if $\mathcal{O}_\varphi$ satisfies the interpretations of
    $\Box \varphi$ and $\Diamond \varphi$ in every topological space in $\mathcal{T}$, then $\varphi$ is derivable in $\mathsf{S4}$.
\end{theorem}
\begin{proof}
    First, we demonstrate that the collection of formulae of the form $\Box \varphi$ satisfies the axioms and rules of $\mathsf{S4}$. 
    The axiom $\Box (\varphi \to \psi) \to (\Box \varphi \to \Box \psi)$ corresponds topologically to the inclusion 
    $\mathrm{int}(\mathcal{O}_{\varphi \to \psi}) \subseteq \mathrm{int}(\mathcal{O}_\varphi) \to \mathrm{int}(\mathcal{O}_\psi)$, 
    which holds because $\mathrm{int}$ preserves logical implication in the sense of set inclusion. The axiom $\Box \varphi \to \varphi$ 
    corresponds to $\mathrm{int}(\mathcal{O}_\varphi) \subseteq \mathcal{O}_\varphi$, which is true because the interior is a subset of the 
    original set. The axiom $\Box \varphi \to \Box \Box \varphi$ corresponds to $\mathrm{int}(\mathcal{O}_\varphi) \subseteq 
    \mathrm{int}(\mathrm{int}(\mathcal{O}_\varphi))$, which holds because $\mathrm{int}$ is idempotent. These interpretations show that 
    $\Box \varphi$ is soundly interpreted as $\mathrm{int}(\mathcal{O}_\varphi)$.

    Similarly, $\Diamond \varphi = \neg \Box \neg \varphi$ implies that $\Diamond \varphi$ corresponds to the closure of the set where 
    $\varphi$ holds, $\mathcal{O}_\varphi$. The axiom $\varphi \to \Diamond \varphi$ corresponds to $\mathcal{O}_\varphi \subseteq 
    \mathrm{cl}(\mathcal{O}_\varphi)$, which holds because any set is contained in its closure. The axiom $\Diamond \Diamond \varphi 
    \to \Diamond \varphi$ corresponds to $\mathrm{cl}(\mathrm{cl}(\mathcal{O}_\varphi)) \subseteq \mathrm{cl}(\mathcal{O}_\varphi)$, 
    which holds because $\mathrm{cl}$ is idempotent. This shows that $\Diamond \varphi$ is soundly interpreted as $\mathrm{cl}
    (\mathcal{O}_\varphi)$.

    To establish the completeness of this interpretation, consider the Heyting algebra of formulae $\mathcal{L}' = \{\Box \varphi \mid 
    \varphi \in \mathcal{L}\}$ with operations defined by $\Box \varphi \wedge \Box \psi = \Box (\varphi \wedge \psi)$, $\Box \varphi 
    \vee \Box \psi = \Box (\varphi \vee \psi)$, and $\Box \varphi \to \Box \psi = \Box \chi$ where $\chi$ satisfies $\Box (\varphi 
    \wedge \chi) \leq \Box \psi$. This satisfies the axioms of a Heyting algebra, as can be verified explicitly by checking associativity, 
    commutativity, and distributivity of the operations. By Stone duality, this Heyting algebra corresponds to a topological space $X$ 
    where each $\Box \varphi$ is uniquely associated with an open set $\mathcal{O}_{\Box \varphi} \subseteq X$. Similarly, the formulae 
    $\mathcal{L}'' = \{\Diamond \varphi \mid \varphi \in \mathcal{L}\}$ form a co-Heyting algebra, where operations are defined dually 
    and correspond to closed sets in the same topological space $X$.

    For any topological space $X$ and any formula $\varphi$, the interpretations $\mathrm{int}(\mathcal{O}_\varphi)$ and 
    $\mathrm{cl}(\mathcal{O}_\varphi)$ satisfy the rules of inference of $\mathsf{S4}$, as shown earlier. Thus, if $\varphi$ is derivable 
    in $\mathsf{S4}$, it is valid in every topological model. Conversely, if $\varphi$ is valid in every topological model, then it holds 
    in the canonical model associated with the Heyting algebra of $\mathcal{L}'$. Since this canonical model satisfies all axioms and 
    rules of $\mathsf{S4}$, $\varphi$ must be derivable in $\mathsf{S4}$. This completes the proof of soundness and completeness.
\end{proof}
Thus, for any $\mathsf{S4}$ theory $\mathcal{L}$, there exists a topological space $X_{\mathcal{L}}$ such that $\mathcal{O}(X_{\mathcal{L}})$
is a Heyting algebra such that $\{\Box \varphi \mid \varphi \in \mathcal{L}\} \simeq \mathcal{O}(X_{\mathcal{L}})$, and $\mathcal{C}(X_{\mathcal{L}})$
is a co-Heyting algebra such that $\{\Diamond \varphi \mid \varphi \in \mathcal{L}\} \simeq \mathcal{C}(X_{\mathcal{L}})$. Conversely, for any 
topological space $X$, there exists an $\mathsf{S4}$ theory $\mathcal{L}_{X}$ such that $\{\Box \varphi \mid \varphi \in \mathcal{L}_X\} \simeq \mathcal{O}(X)$, and 
$\{\Diamond \varphi \mid \varphi \in \mathcal{L}\} \simeq \mathcal{C}(X)$. Since any Heyting algebra is isomorphic to the lattice of opens of some 
space, and the lattice of open sets of any space can be embedded in a Heyting algebra, and similar for co-Heyting algebras, we can take the soundness 
and completeness of $\mathsf{S4}$ with respect to topological spaces as saying that we may interpret any Heyting algebra as an algebra of \textit{necessary}
propositions for a propositional $\mathsf{S4}$ theory, and any co-Heyting algebra may be interpreted as an algebra of \textit{possible} propositions.
This tracks with the interpretations of both intuitionistic and dual intuitionistic logic spelled out earlier in this paper. A \textit{necessary}
proposition is one that is always provable (i.e a tautology), whereas a \textit{possible} proposition is not disprovable, or not refutable. 
As a corollary of the soundness and completeness theorem, one can also prove an interesting result showing that worlds in Kripke frames correspond 
to points of a topological space $X$ modeling an $\mathsf{S4}$ theory $\mathcal{L}$. We will prove only one direction, as the other direction is 
largely symmetric.

\begin{proposition}
    Let $X$ be a topological space corresponding to an $\mathsf{S4}$ language $\mathcal{L}$. The points of $X$ are the worlds $w$ in a Kripke 
    model $(W, R, \models)$ sound for $\mathcal{L}$.
\end{proposition}

\begin{proof}
    We assume that the $\mathsf{S4}$ language $\mathcal{L}$ is interpreted within the topological space $X$, with formulas in $\mathcal{L}$
    corresponding to subsets of $X$. The mapping to subsets is done in the same way done in the proof of the soundness and completeness theorems 
    above.
    
    In Kripke semantics, a frame is a pair $(W, R)$ where $W$ is a set of worlds and $R \subseteq W \times W$ is a binary relation 
    indicating accessibility between worlds. Satisfaction in a Kripke model is defined such that a world $w \in W$ satisfies 
    $\Box \varphi$ if all $R$-accessible worlds satisfy $\varphi$, and $w$ satisfies $\Diamond \varphi$ if at least one $R$-accessible 
    world satisfies $\varphi$.
    To relate $X$ to a Kripke frame, we treat the points of $X$ as the worlds in the frame, so that $W = X$. Accessibility between worlds
    is topologically induced. Specifically, for $x, y \in X$, we say $x R y$ if $y$ lies in every open set containing $x$. This notion
    of accessibility aligns with the interpretation of modal operators in $X$, as established in the proof of soundness and completeness with 
    respect to topological semantics. 
    To establish the soundness of this Kripke model, we verify that $X$ satisfies the conditions of the Kripke semantics. For 
    $\Box \varphi$, we note that $\Box \varphi$ corresponds to $\mathrm{int}(\mathcal{O}_{\varphi})$, so $x \models \Box \varphi$ if and only
    if $x \in \mathrm{int}(\mathcal{O}_{\varphi})$, with notation for open sets carried over from the proof of the soundness and completeness theorem.
    This holds if and only if every point $y$ accessible from $x$ satisfies $\varphi$, as required by Kripke semantics. Similarly, 
    $\Diamond \varphi$ corresponds to $\mathcal{O}_{\varphi}^c$, so $x \models \Diamond \varphi$ if and only if $x \in \mathcal{O}_{\varphi}^c$. This is true if and only if there exists an accessible point $y$ such that 
    $y \models \varphi$, again in agreement with Kripke semantics.
    Thus, the points of $X$ are the worlds of a Kripke frame $(X, R)$, where $R$ is defined by the topological structure, and this frame 
    is sound for the modal language $\mathcal{L}$. This can be extended to a Kripke model $(X, R, \models)$ with the satisfaction relation 
    defined by accessibility as above.
\end{proof}

Thus, to complete the translation of intuitionistic and dual intuitionistic notions to the language of $\mathsf{S4}$, the prime filters of a
Heyting algebra $H$, or a co-Heyting algebra $H^{op}$, naturally carry the structure of an $\mathsf{S4}$ Kripke frame. Indeed, the operation of 
quotienting by a filter, as described in section 1, can be seen as ``collapsing'' possible worlds in logical space onto one another -- giving a vivid 
and clear picture of the algebraic meaning of quotients on Heyting and co-Heyting algebras. Conceptually, we may summarize all of our findings thus far with the following diagram.
\vspace{2em}
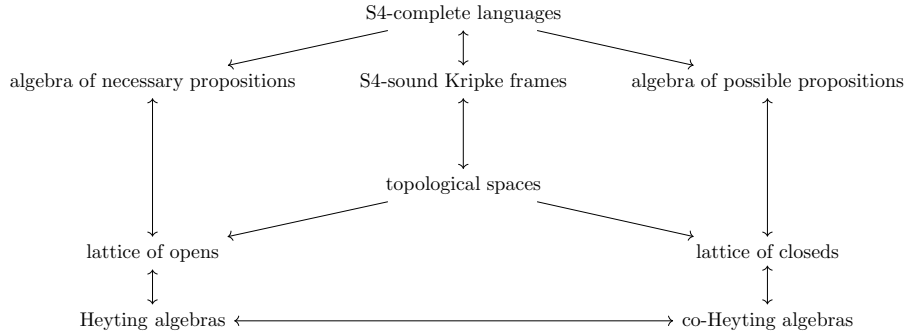
\begin{center}
\adjustbox{scale=0.75,center}{%
    \begin{tikzcd}
        & {\textrm{S4-complete languages}} \\
        {\textrm{algebra of necessary propositions}} & {\textrm{S4-sound Kripke frames}} & {\textrm{algebra of possible propositions}} \\
        \\
        & {\textrm{topological spaces}} \\
        {\textrm{lattice of opens}} && {\textrm{lattice of closeds}} \\
        {\textrm{Heyting algebras}} && {\textrm{co-Heyting algebras}}
        \arrow[from=1-2, to=2-1]
        \arrow[from=1-2, to=2-3]
        \arrow[tail reversed, from=2-1, to=5-1]
        \arrow[tail reversed, from=2-2, to=1-2]
        \arrow[tail reversed, from=2-2, to=4-2]
        \arrow[tail reversed, from=2-3, to=5-3]
        \arrow[from=4-2, to=5-1]
        \arrow[from=4-2, to=5-3]
        \arrow[tail reversed, from=5-1, to=6-1]
        \arrow[tail reversed, from=5-3, to=6-3]
        \arrow[tail reversed, from=6-1, to=6-3]
    \end{tikzcd}
}
\end{center}
\vspace{2em}
This suggests we can try to characterize some of the stranger operations and properties of dual intuitionistic logic in terms of the Kripke semantics for $\mathsf{S4}$. 
Indeed, we will now construct simple $\mathsf{S4}$ Kripke models to do this.
\begin{example}
    A Kripke model for $\Diamond p \wedge \Diamond \neg p$. 
\end{example}
\begin{proof}[Example]
    Consider a Kripke model $\mathcal{M} = (W, R, V)$, where $W = {w_0, w_1, w_2}$, $R$ is the accessibility relation, and $V$ is the valuatio
    n function (this is a special case of the modeling relation $\models$). We define the accessibility relation $R$ so that it satisfies the
    $\mathsf{S4}$ axioms: suppose $w R w$ for all $w \in W$, so $R$ is reflexive, and suppose $w_0 R w_1$ and $w_1 R w_2$ imply $w_0 R w_2$,
    so $R$ is transitive. Now suppose that $w_0 R w_1$ and $w_0 R w_2$, but that $w_1$ and $w_2$ are not accessible to each other. The 
    valuation function $V$ assigns $V(p) = {w_1}$, so $p$ is true in $w_1$ and false in $w_0$ and $w_2$.

    In this model, $w_0 \models \Diamond p$, since $w_0 R w_1$ and $w_1 \models p$. Simultaneously, $w_0 \models \Diamond \neg p$, because
    $w_0 R w_2$ and $w_2 \models \neg p$. Hence, $w_0 \models \Diamond p \wedge \Diamond \neg p$, demonstrating that this configuration
    satisfies the $\mathsf{S4}$ axioms while modeling the formula $\Diamond p \wedge \Diamond \neg p$.
\end{proof}
It is obvious that $\Diamond p \wedge \Diamond \neg p$ corresponds via topological semantics to a dual intuitionistic statement of the form 
$a \wedge {\sim} a$. Thus, in a sense, this example shows that the paraconsistency of dual intuitionistic logic is ``global'' in a sense: true 
contradictions arise from propositions that are inconsistent at the \textit{global} level, even if they are perfectly consistent at the local 
level. Note that this example above would \textit{not} work in $\mathsf{S5}$: the symmetry of $R$ would be required in that case, which would 
break the construction. We now move on to another example where Kripke semantics helps us build an intuition for the weirdness of dual intuitionistic 
logic.
\begin{example}
    Demystifying the $\leftarrow$ operation.
\end{example}
\begin{proof}[Example]
    Consider a Kripke model $\mathcal{M} = (W, R, V)$ with $W = {w_0, w_1, w_2}$, $R = {(w_0, w_1), (w_0, w_2), (w_1, w_1), (w_2, w_2)}$, 
    and $V(p) = {w_1}$, $V(q) = {w_2}$. The accessibility relation $R$ satisfies reflexivity since $(w_1, w_1) \in R$ and 
    $(w_2, w_2) \in R$, ensuring $w R w$ for all $w \in W$. It satisfies transitivity because if $(w_0, w_1) \in R$ and $(w_1, w_1) \in R$, 
    then $(w_0, w_1) \in R$, and similarly for other pairs. Thus $\mathcal{M}$ corresponds to an $\mathsf{S4}$ theory. At $w_0$, $w_0 
    \models \Diamond p$ since $w_1 \models p$ and $w_0 R w_1$, and $w_0 \models \Diamond q$ since $w_2 \models q$ and $w_0 R w_2$. 
    At $w_1$, $w_1 \models p \wedge \neg q$, and at $w_2$, $w_2 \models q \wedge \neg p$. The formula $w \models \chi$ represents the 
    strongest condition on $w$ such that, whenever $w \models \chi$ and $w \models \varphi$, it must follow that $w \models \psi$, for
    arbitrary modal $\varphi$ and $\psi$. In modal terms, this ensures that $\chi$ captures exactly the worlds where $\psi$ holds or 
    $\varphi$ is false. At $w_0$, $w_0 \models \chi$ since $\Diamond p$ and $\Diamond q$ are both satisfied. At $w_1$, $\chi$ is false 
    since $p$ is true and $q$ is false. At $w_2$, $\chi$ is true since $q$ is true.
\end{proof}
For arbitrary dual intuitionistic propositions $\varphi$ and $\psi$, is evident that $\varphi \leftarrow \psi$ corresponds to the strongest 
dual intuitionistic proposition ensuring that one can conclude $\psi$ without necessarily concluding $\varphi$. In the Kripke model, this is represented by a modal formula true in
all worlds where $\psi$ holds or $\varphi$ does not. Using the dualities we have developed to associate dual intuitionistic propositions $\varphi$ and $\psi$ with $\Diamond p$
and $\Diamond q$, respectively, the model reveals how co-implication relies on the interaction between accessibility and possibility. In a 
sense, this example shows that $\varphi \leftarrow \psi$ encodes a ``fallback" relationship: it is globally consistent as long as no
accessible world simultaneously forces $\varphi$ and excludes $\psi$. Note that this structure leverages the asymmetry of $\mathsf{S4}$, 
where accessibility is reflexive and transitive but not symmetric. 

With this, we have done a large part of what we set out to do in this paper. It is clear that $\mathsf{S4}$ subsumes both intuitionistic and 
dual intuitionistic logic in an important way: intuitionistic logic provides the backbone for reasoning about \textit{necessary} propositions 
in a constructive, enumerative way. On the other hand, dual intuitionistic logic provides a language in which we may reason about propositions 
that are \textit{possible}, including propositions that when conceptualized as sets contains worlds that prima facie ``contradict'' one another.
However, as the $\mathsf{S4}$ axioms have showed us, such contradiction is usually of a \textit{global} rather than a local nature. Thus, from 
this point of view, much of the mystery underlying the paraconsistency of dual intuitionistic logic simply vanishes. 

\section*{Concluding remarks}
We have come a long way in this paper, and the journey warrants reflection. The first key insight that the author would like to stress is more 
technical: the expressive power and decidability of dual intuitionistic logic is directly tied to the analogous properties for $\mathsf{S4}$. 
Thus, since $\mathsf{S4}$ can be embedded into a fully decidable subset of (classical) first-order predicate logic, the expressive power of the 
propositional variants of dual intuitionistic logic (also intuitionistic logic) cannot exceed that of first-order logic, or indeed even that of 
$\mathsf{S4}$ (as both logics are, in a sense, subsumed by it). While this removes much of the \textit{mystqiue} surrounding the paraconsistency 
of propositional dual intuitionistic logic, it does make it much easier to deal with than many other paraconsistent logics. While logics like 
Priest's Logic of Paradox, or Kleene's 3 valued logic have simple algebraic and possible worlds semantics, they do not have the nice and ``clean''
relationship to topological spaces and $\mathsf{S4}$ that dual intuitionistic logic does (see Priest \cite{priest_2008_an} or Font \cite{josepmariafont_2016_abstract} for more discussion of this). Perhaps this indicates that this logic is deal for dealing 
with inconsistencies arising in philosophical situations relating to impossible worlds, properties, and a myriad of other metaphysical issues: 
while we can use paraconsistent \textit{language} from the perspective of co-Heyting algebras, we can always translate proof-theoretic or semantic 
concerns back into the philosophical \textit{lingua franca} of $\mathsf{S4}$. This means dual intuitionistic logic is a paraconsistent logic that 
is nonetheless rooted in a ``friendly face'' who makes the whole idea of inconsistency much more tractable.

Moreover, in reading this paper, one might have noted that the discussion never left the \textit{propositional} case: indeed, it was propositional 
intuitionistic logic, dual intuitionistic logic, and $\mathsf{S4}$ that were discussed throughout. It is interesting to ponder how the generalization 
of these relationships and dualities to higher-order intuitionistic and dual intuitionistic logic would work. It is the author's suspicion that this 
could be done by synthesizing the extensive literature available on the sheaf-theoretic semantics for higher-order $\mathsf{S4}$ \textit{modal} logic 
(as given by Awodey et al. and Kishida \cite{awodey_2014_topos} \cite{kishida_2011_neighborhoodsheaf}) with concepts from topos theory, see also Kishida's essay in 
Landry \cite{landry_2017_categories}. Sheaves and topoi are the key technical tools to generalizing Stone-type 
spectral spaces as used in this paper to higher-order contexts, and in the case of \textit{Heyting algebras}, the generalization is well-understood:
the global sections of the subobject classifier of every elementary topos are isomorphic to a Heyting algebra, and so via construction of an appropriately
chosen Grothendieck topology and some specification of an ``internal type theory'', every elementary topos can be realized as a model of higher 
order intuitionistic logic (a standard reference for this is MacLane-Moerdijk \cite{maclane_2012_sheaves}). In relation to the discussion in this paper, this construction is possible because the Grothendieck topologies defining 
elementary topoi are defined in terms of \textit{open coverings}, which are directly analogous to the open sets of a topological space. It is less clear,
however, how one might define topoi with some notion of \textit{closed covering}. Work in providing higher-order categorical semantics for 
paraconsistent logics has so far been focused on the logic of \textit{complement} topoi, which is a ``topos in reverse'' that does not dualize 
all constructions fully -- this does not yield higher-order dual intuitionistic logic. Moreover, in his dissertation \textit{Closed Set Categories in 
Categories}, William James claims that when the appropriate kind of semantic dualization \textit{is} carried out, the resulting analogue to the 
subobject classifer (which is a quotient object classifier) remains a Heyting algebra rather than becoming a co-Heyting algebra \cite{james_1996_closed}. 

A higher-order generalization of the contents of this paper would thus solve major problems related to the categorical semantics of higher-order 
paraconsistent logics in general. Moreover, given the close relationship that higher-order constructive logic as modeled by topoi share with 
notions in set theory, it is quite possible that such advances would shed light on phenomena like forcing in set theory. Indeed, if there is a possible 
world that exists at the boundary of two contradictory propositions, in a sense one could say that world is ``independent'' of the others. It is the 
author's intuition that this dual intuitionistic concept could share a serious relationship with the Kripke-Joyal semantics for forcing that sheds 
light on the topos-theoretic view of the universe of sets (see Reyes and Zolfaghari \cite{reyes_1996_biheyting} for more discussion of this). Obviously, possible applications to theoretical computer science (via the theory of the effective 
topos) are very much feasible too. All in all, it is the hope of the author that this paper convinced you that dual intuitionistic logic is an 
interesting, very tractable, and very elegant paraconsistent logic, worthy of deeper investigation and further study.

\section*{Acknowledgements}
This research was done by the author through the benefit of the Jejurikar Summer Scholars Fund. The author would like to thank Professor David 
Smyth for his ceaseless and enthusiastic mentorship throughout the time this research was conducted, and would also like to thank the Office 
of Scholar Development at Tufts University, especially Dr. Anne Moore, for providing him the opportunity to produce this research as a part of the 
Tufts Summer Scholars program. Thanks are also due to Dr. Michael Jahn and Professor Emeritus Anselm Blumer for their detailed proofreading and 
valuable comments.

\nocite{*}
\printbibliography

@phdthesis{james_1996_closed,
  author = {James, William},
  month = {08},
  title = {Closed set logic in categories},
  year = {1996}
}

@article{lawvere_1991_intrinsic,
  author = {Lawvere, F. William},
  publisher = {Springer Heidelberg},
  title = {Intrinsic Co-Heyting Boundaries and the Leibniz Rule in Certain Toposes},
  volume = {LNM 1488},
  year = {1991},
  organization = {Category Theory - Proceedings of the International Conference Held in Como 1990}
}

@book{gratzer_2017_lattice,
  author = {Gratzer, George},
  publisher = {Birkhauser Verlag Ag},
  title = {Lattice Theory.},
  year = {2017}
}

@book{maclane_2012_sheaves,
  author = {MacLane, Saunders and Ieke Moerdijk},
  month = {12},
  publisher = {Springer Science and Business Media},
  title = {Sheaves in Geometry and Logic},
  year = {2012}
}

@book{johnstone_1982_stone,
  author = {Johnstone, Peter T},
  publisher = {Cambridge University Press},
  title = {Stone Spaces},
  year = {1982}
}

@book{landry_2017_categories,
  author = {Landry, Elaine M},
  publisher = {Oxford University Press},
  title = {Categories for the working philosopher},
  year = {2017}
}

@book{fagin_2011_reasoning,
  author = {Fagin, Ronald and others},
  publisher = {Cambridge, Mass. MIT Press},
  title = {Reasoning about knowledge},
  year = {2011}
}

@book{priest_2008_an,
  author = {Priest, Graham},
  month = {04},
  publisher = {Cambridge University Press},
  title = {An Introduction to Non-Classical Logic},
  year = {2008}
}

@article{kishida_2011_neighborhoodsheaf,
  author = {Kishida, Kohei},
  month = {11},
  pages = {129-143},
  title = {Neighborhood-Sheaf Semantics for First-Order Modal Logic},
  doi = {10.1016/j.entcs.2011.10.011},
  volume = {278},
  year = {2011},
  journal = {Electronic Notes in Theoretical Computer Science}
}

@article{awodey_2014_topos,
  author = {Awodey, Steve and Kishida, Kohei and Kotzsch, Hans-Christoph},
  title = {Topos Semantics for Higher-Order Modal Logic},
  doi = {10.2143/LEA.228.0.3078176},
  volume = {57},
  year = {2014},
  journal = {Logique et Analyse}
}

@book{josepmariafont_2016_abstract,
  author = {Josep Maria Font},
  month = {04},
  publisher = {College Publications},
  title = {Abstract Algebraic Logic. an Introductory Textbook},
  year = {2016}
}

@article{yaroslavshramko_2005_dual,
  author = {Yaroslav Shramko},
  month = {08},
  pages = {347-367},
  publisher = {Springer Science+Business Media},
  title = {Dual Intuitionistic Logic and a Variety of Negations: The Logic of Scientific Research},
  doi = {10.1007/s11225-005-8474-7},
  volume = {80},
  year = {2005},
  journal = {Studia Logica}
}

@book{riehl_2016_category,
  author = {Riehl, Emily},
  publisher = {Dover Publications, Cop},
  title = {Category theory in context},
  year = {2016}
}

@book{paoloaluffi_2021_algebra,
  author = {Paolo Aluffi},
  month = {11},
  publisher = {American Mathematical Soc.},
  title = {Algebra: Chapter 0},
  year = {2021}
}

@book{munkres_2018_topology,
  author = {Munkres, James R},
  publisher = {Pearson},
  title = {Topology},
  year = {2018}
}

@article{reyes_1996_biheyting,
  author = {Reyes, Gonzalo E. and Zolfaghari, Houman},
  month = {02},
  pages = {25-43},
  title = {Bi-Heyting algebras, toposes and modalities},
  doi = {10.1007/bf00357841},
  volume = {25},
  year = {1996},
  journal = {Journal of Philosophical Logic}
}

@article{BEZHANISHVILI2019403,
title = {A semantic hierarchy for intuitionistic logic},
journal = {Indagationes Mathematicae},
volume = {30},
number = {3},
pages = {403-469},
year = {2019},
issn = {0019-3577},
doi = {https://doi.org/10.1016/j.indag.2019.01.001},
author = {Guram Bezhanishvili and Wesley H. Holliday},
}

\end{document}